\documentclass[10pt]{article}

\usepackage[nocompress]{cite}
\usepackage{amsmath}
\usepackage{amsthm}
\usepackage{amssymb}
\usepackage{latexsym}
\usepackage{mathdots}
\usepackage{tikz}
\usepackage{titlefoot}
\usepackage{url}
\usepackage{titlesec}
\usepackage{cancel}
\usepackage{rotating}
\usepackage[small, it]{caption}

\usepackage[labelformat=simple,skip=-2pt]{subcaption}

\newcommand{\Av}{\operatorname{Av}}
\newcommand{\ds}{\displaystyle}
\newcommand{\pa}[1]{\left({#1}\right)}
\newcommand{\f}[2]{\ds\frac{{#1}}{{#2}}}
\newcommand{\CC}{\mathcal{C}}

\newcommand{\GG}{\mathcal{G}}

\newcommand{\LL}{\mathcal{L}}
\newcommand{\calSS}{\mathcal{S}}
\renewcommand{\phi}{\varphi}
\newcommand{\Geom}{\operatorname{Geom}}
\newcommand{\res}[1]{\big{|}_{#1}}
\newcommand{\Simples}{\operatorname{Si}}
\newcommand{\ssm}{\smallsetminus}
\newcommand{\io}{\Av(321,2341,3412,4123)}

\usepackage{color}
\definecolor{light-gray}{gray}{0.8}
\definecolor{dark-gray}{gray}{0.3}
\definecolor{darkblue}{rgb}{0, 0, .4}

\usepackage[bookmarks]{hyperref}
\hypersetup{
        colorlinks=true,
        linkcolor=darkblue,
        anchorcolor=darkblue,
        citecolor=darkblue,
        urlcolor=darkblue,
        pdfpagemode=UseThumbs,
        pdftitle={The Enumeration of Permutations Avoiding 3124 and 4312},
        pdfsubject={Combinatorics},
        pdfauthor={Jay Pantone},
}

\newtheorem{theorem}{Theorem}[section]
\newtheorem{lemma}[theorem]{Lemma}

\newcounter{todocounter}

\long\def\symbolfootnote[#1]#2{\begingroup%
\def\thefootnote{\fnsymbol{footnote}}\footnote[#1]{#2}\endgroup}

\makeatletter
\newcommand{\Rm}[1]{\expandafter\@slowromancap\romannumeral #1@}
\newfont{\footsc}{cmcsc10 at 8truept}
\newfont{\footbf}{cmbx10 at 8truept}
\newfont{\footrm}{cmr10 at 10truept}
\renewenvironment{abstract}%
                {
                  \begin{list}{}%
                     {\setlength{\rightmargin}{1in}%
                      \setlength{\leftmargin}{1in}}%
                   \item[]\ignorespaces\begin{small}}%
                 {\end{small}\unskip\end{list}}

\datefoot{\today}
\amssubj{05A05, 05A15}
\keywords{algebraic generating function, permutation class, simple permutation, geometric grid class}

\newpagestyle{main}[\small]{
        \headrule
        \sethead[\usepage][][]
        {\sc The Enumeration of Permutations Avoiding 3124 and 4312}{}{\usepage}}

\title{\sc{The Enumeration of Permutations\\Avoiding 3124 and 4312}}
\author{Jay Pantone\\[-0.25ex]
\small Department of Mathematics\\[-0.5ex]
\small University of Florida\\[-0.5ex]
\small Gainesville, Florida\\[-1.5ex]
}

\titleformat{\section}
        {\large\sc}
        {\thesection.}{1em}{}   

\date{}

\newcommand{\OEISlink}[1]{\href{http://oeis.org/#1}{#1}}
\newcommand{\OEISref}{\href{http://oeis.org/}{OEIS}~\cite{oeis}}
\newcommand{\OEIS}[1]{sequence \OEISlink{#1} in the \OEISref}

\begin{document}
\maketitle

\pagestyle{main}

\begin{abstract}
We find the generating function for the class of all permutations that avoid the patterns 3124 and 4312 by showing that it is an inflation of the union of two geometric grid classes.
\end{abstract}

\section{Preliminaries}

We consider permutations in \emph{one-line notation} so that a permutation of length $n$ is treated as a linear ordering of the symbols of $\{1, 2, \ldots, n\}$. 
For permutations $\tau$ and $\pi$ with
lengths $k$ and $n$ respectively, we write $\tau \leq \pi$ (or ``$\pi$ contains $\tau$'') if there is a set of indices $1 \leq i_1 < i_2 < \cdots < i_k \leq n$ such that the sequence $\pi(i_1), \pi(i_2), \ldots, \pi(i_k)$ is in the same relative order as $\tau$. For example, $4312 \leq 4756231$ because the entries $7523$ in $4756231$ have the same relative order as $4312$. 

A \emph{permutation class} is a set of permutations that is closed downward under this order, i.e., if $\CC$ is a permutation class and if $\pi \in \CC$ and $\tau \leq \pi$, then $\tau \in \CC$. 
We can describe a permutation class by specifying a list of permutations that it avoids. As an example, the set of all strictly increasing permutations may be denoted $\Av(21)$. This paper studies the class of permutations that avoid both $3124$ and $4312$, denoted $\Av(3124,4312)$. In particular, we derive its generating function $\sum a_nx^n$, where $a_n$ is the number of permutations of length $n$ in this class. 


A permutation class is called a \emph{2$\times$4 class} if the minimal elements in the set of permutations not in the class consist exactly of two permutations of length four. Up to symmetries of the permutation containment order --- the group generated by the symmetries reverse, complement, and group-theoretic inverse --- there are 56 different 2$\times$4 classes. Some of these have the same enumeration; it has been shown that there are precisely 38 different enumerations for the 2$\times$4 classes \cite{bona:perm-class-smooth, kremer:forbiddenps, kremer:finite, le:wilf-classes-pairs, kremer:forbidden}. This paper will bring the number of different enumerations that have been found to 29. See Wikipedia~\cite{wiki:enum} for a list of currently known enumerations.

Albert, Atkinson, and Vatter~\cite{albert:inflations-case-studies} enumerated three of the 2$\times$4 classes by studying inflations of geometric grid classes. We show that $\Av(3124,4312)$ is also amenable to the same techniques despite its significantly more complicated structure. Furthermore, in Section~\ref{section:other-classes} we show that these methods cannot be applied to any other unenumerated 2$\times$4 classes.

\section{Simple Permutations and Inflations}

An \emph{interval} of a permutation is a nonempty set of consecutive indices $\{i, i+1, \ldots, i+k\}$ such that the entries $\{\pi(i), \pi(i+1), \ldots, \pi(i+k)\}$  form a set of consecutive integers. A permutation of length $n$ is said to be \emph{simple} if its only intervals are those of length 1 and $n$. In Figure~\ref{figure:interval}, the permutation $31468572$ is not simple because it contains a nontrivial interval at the indices $\{4,5,6,7\}$ with entries $\{5,6,7,8\}$, and the permutation $63814725$ is simple because it contains no nontrivial intervals. We say that the permutations $1$, $12$, and $21$ are simple. We use $\Simples(\CC)$ to denote the set of simple permutations of a class $\CC$.

\begin{figure}
\begin{center}
	\begin{tikzpicture}[scale=.3]
		\draw[ultra thick] ({1-.0355},0)--(8.0355,0);
		\draw[ultra thick] (0,{1-.0355})--(0,8.0355);
		\foreach \x in {1,...,8} {
			\draw[thick] (\x,0)--(\x,-.5);
			\draw[thick] (0,\x)--(-.5,\x);
		}
		\draw[fill=black] (1,3) circle (5pt);
		\draw[fill=black] (2,1) circle (5pt);
		\draw[fill=black] (3,4) circle (5pt);
		\draw[fill=black] (4,6) circle (5pt);
		\draw[fill=black] (5,8) circle (5pt);
		\draw[fill=black] (6,5) circle (5pt);
		\draw[fill=black] (7,7) circle (5pt);
		\draw[fill=black] (8,2) circle (5pt);
		\draw[thick] (3.5, 4.5) rectangle (7.5, 8.5);
		\node at (4.5,-1.5) {$\tau = 31468572$};
	\end{tikzpicture}
	\qquad\qquad\qquad\qquad
	\begin{tikzpicture}[scale=.3]
		\draw[ultra thick] ({1-.0355},0)--(8.0355,0);
		\draw[ultra thick] (0,{1-.0355})--(0,8.0355);
		\foreach \x in {1,...,8} {
			\draw[thick] (\x,0)--(\x,-.5);
			\draw[thick] (0,\x)--(-.5,\x);
		}
		\draw[fill=black] (1,6) circle (5pt);
		\draw[fill=black] (2,3) circle (5pt);
		\draw[fill=black] (3,8) circle (5pt);
		\draw[fill=black] (4,1) circle (5pt);
		\draw[fill=black] (5,4) circle (5pt);
		\draw[fill=black] (6,7) circle (5pt);
		\draw[fill=black] (7,2) circle (5pt);
		\draw[fill=black] (8,5) circle (5pt);
		\node at (4.5,-1.5) {$\rho = 63814725$};
	\end{tikzpicture}
\end{center}
\caption{The permutation $\tau$ is not simple because it contains a nontrivial interval. The permutation $\rho$ is simple because it contains no nontrivial intervals.}
\label{figure:interval}
\end{figure}
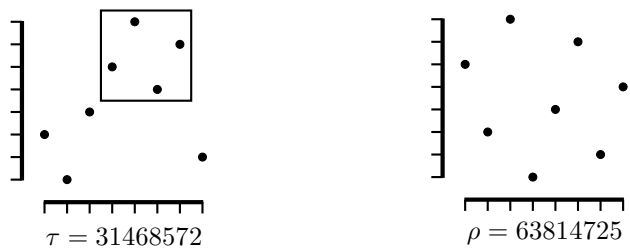

Simple permutations are the foundation from which we can build all other permutations. The \emph{inflation} of a permutation $\pi$ of length $n$ by a sequence of nonempty permutations $\tau_1, \ldots, \tau_n$, denoted $\pi\left[\tau_1, \ldots, \tau_n\right]$, is the permutation
that results from taking each entry $\pi(i)$ and replacing it with an interval that is order-isomorphic to $\tau_i$ such that the intervals themselves are order-isomorphic to $\pi$. For example
	\[3142\left[231,21,123,1\right] = 564\;21\;789\;3.\]
	
Inflations of the permutations $12$ and $21$ are given their own names. The \emph{sum} $\sigma \oplus \tau$ is defined to be the inflation $12\left[\sigma, \tau\right]$, while the \emph{skew sum} $\sigma \ominus \tau$ is defined to be the inflation $21\left[\sigma, \tau\right]$. A permutation $\pi$ is said to be \emph{sum decomposable} if $\pi = \sigma \oplus \tau$ for some $\sigma$ and $\tau$; otherwise $\pi$ is said to be \emph{sum indecomposable}. Analogously, a permutation $\pi$ is said to be \emph{skew decomposable} if $\pi = \sigma \ominus \tau$ for some $\sigma$ and $\tau$, and is otherwise \emph{skew indecomposable}.

Each element of a permutation class can be expressed as an inflation of a unique simple permutation in the class, as given by the following lemma.

\begin{lemma}[Albert and Atkinson~\cite{albert:simple-permutations}]\label{lemma:decomposition}
For every permutation $\pi$ there is a unique simple permutation $\sigma$ such that $\pi=\sigma[\tau_1,\ldots,\tau_k]$.  When $\sigma \neq 12$ and $\sigma \neq 21$, the intervals of $\pi$ that correspond to $\tau_1,\ldots,\tau_k$ are uniquely determined.  When $\sigma=12$ (respectively, $\sigma=21$), the intervals are unique so long as we require the first of the two intervals to be sum (respectively, skew) indecomposable.
\end{lemma}

\section{Geometric Grid Classes}
\label{section:ggc-info}

Let $M$ be a matrix whose entries are from the set $\left\{-1, 0, +1, \bullet\right\}$. To be consistent with plots of permutations, we index matrices with Cartesian coordinates: first by column (left to right) and then by row (bottom to top), starting at 1. Define the \emph{standard figure} of $M$ to be the drawing on the Cartesian plane that consists of:\begin{itemize}
	\item the line segment from $(k-1,\;\ell-1)$ to $(k,\; \ell)$, if $M_{k,\ell} = +1$,
	\item the line segment from $(k-1,\;\ell)$ to $(k,\;\ell-1)$, if $M_{k,\ell} = -1$, and
	\item the point $\pa{k-\frac{1}{2},\; \ell - \frac{1}{2}}$, if $M_{k,\ell} = \bullet$.
\end{itemize}

Informally, each entry of the matrix becomes a cell that is either empty or contains an increasing line segment, a decreasing line segment, or a single point, depending on the corresponding matrix entry. Herein, we require that if a cell contains a $\bullet$, then the remaining cells in its row and its column are empty.

A permutation class $\CC$ is a \emph{geometric grid class} if there exists a matrix $M$ such that every permutation in $\CC$ can be drawn on the standard figure of $M$ by placing entries anywhere on the line segments, with at most one entry placed in each $\bullet$ cell. In this case, we say that $\CC = \Geom(M)$. Some well-studied geometric grid classes are, for example, those permutations that can be drawn on an \textsf{X}~\cite{waton:phd, vatter:circle} and those permutations that can be drawn on a diamond~\cite{waton:phd, elizalde:x-class}. See Figure~\ref{figure:ggc-ex}.

\begin{figure}
\begin{center}
	\begin{tikzpicture}[scale=1.3]
		\clip (-0.005,-0.005) rectangle (2.005,2.005);
		\draw[thick] (0,0) rectangle (2,2);
		\draw[thick] (1,0)--(1,2);
		\draw[thick] (0,1)--(2,1);
		\draw[line width=.43mm] (0,0)--(2,2);
		\draw[line width=.43mm] (0,2)--(2,0);
		\draw[fill] (.2,1.8) circle (.06) node[below=2pt] {8};
		\draw[fill] (.4,.4) circle (.06) node[below=2pt] {2};
		\draw[fill] (.6,1.4) circle (.06) node[above=2pt] {7};
		\draw[fill] (.8,1.2) circle (.06) node[above=2pt] {5};
		\draw[fill] (1.1,.9) circle (.06) node[below=2pt] {4};
		\draw[fill] (1.3,1.3) circle (.06) node[left=3pt, above=2pt] {6};
		\draw[fill] (1.5,.5) circle (.06) node[above=2pt] {3};
		\draw[fill] (1.7,.3) circle (.06) node[above=2pt] {1};
		\draw[fill] (1.9,1.9) circle (.06) node[left=1pt, below=2pt] {9};
	\end{tikzpicture}
	\qquad\qquad\qquad\qquad
	\begin{tikzpicture}[scale=1.3]
		\clip (-0.005,-0.005) rectangle (2.005,2.005);
		\draw[thick] (0,0) rectangle (2,2);
		\draw[thick] (1,0)--(1,2);
		\draw[thick] (0,1)--(2,1);
		\draw[line width=.43mm] (0,1)--(1,2)--(2,1)--(1,0)--cycle;
		\draw[fill] (.2,1.2) circle (.06) node[above=2pt] {5};
		\draw[fill] (.4,1.4) circle (.06) node[above=2pt] {7};
		\draw[fill] (.6,1.6) circle (.06) node[above=2pt] {8};
		\draw[fill] (.8,.2) circle (.06) node[above=2pt] {2};
		\draw[fill] (1.1,.1) circle (.06) node[above=2pt] {1};
		\draw[fill] (1.3,1.7) circle (.06) node[right=5pt, above=-1pt] {9};
		\draw[fill] (1.5,.5) circle (.06) node[below=2pt] {3};
		\draw[fill] (1.7,1.3) circle (.06) node[above=2pt] {6};
		\draw[fill] (1.9,.9) circle (.06) node[below=2pt] {4};
	\end{tikzpicture}
\end{center}
\caption[Geometric Grid Class Examples]{The permutation $827546319$ is in the class of permutations that can be drawn on an \textsf{X}, denoted $\Geom\pa{\begin{array}{rr}-1&1\\1&-1\end{array}}$. The permutation $578219364$ is in the class of permutations that can be drawn on a diamond, denoted $\Geom\pa{\begin{array}{rr}1&-1\\-1&1\end{array}}$.}
\label{figure:ggc-ex}
\end{figure}
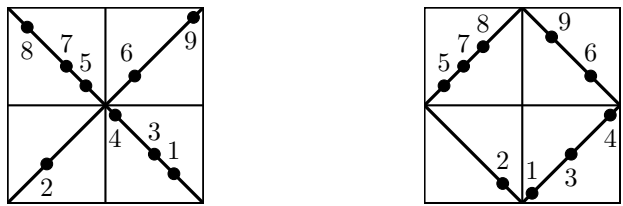

A \emph{consistent orientation} of a geometric grid class is a way of assigning a direction to each line segment in the standard figure such that in any row either all arrows point upward or all arrows point downward and in any column either all arrows point leftward or all arrows point rightward. The \emph{base point} of a cell is the beginning endpoint of its directed line segment, if it has one. 


Figure~\ref{figure:our-classes} shows the three geometric grid classes studied in this paper,
along with 
the consistent orientations that we use for each. It also shows the cell alphabet for each class, which is described below.

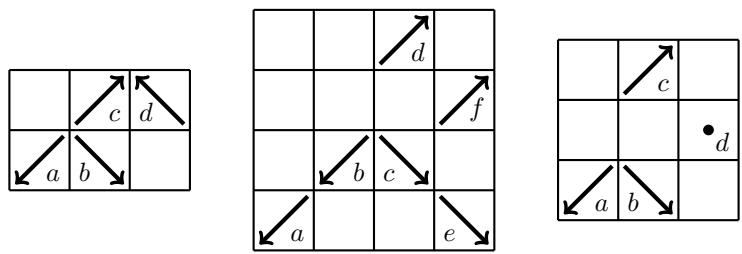
\begin{figure}
\begin{center}
	\begin{tikzpicture}[scale=.8, baseline=(current bounding box.center)]
		\foreach \i in {0,1,2}{
			\draw[thick] (0,\i)--(3,\i);
			\draw[thick] (\i,0)--(\i,2);
		}
		\draw[thick] (3,0)--(3,2);
		\draw[ultra thick,->] (.9,.9)--(.1,.1);
		\draw[ultra thick,<-] (1.9,1.9)--(1.1,1.1);
		\draw[ultra thick,<-] (2.1,1.9)--(2.9,1.1);
		\draw[ultra thick,->] (1.1,.9)--(1.9,.1);
		\node[above left] at (1,0) {$a$};
		\node[above right] at (1,0) {$b$};
		\node[above left] at (2,1) {$c$};
		\node[above right] at (2,1) {$d$};
	\end{tikzpicture}
	\qquad
	\begin{tikzpicture}[scale=.8, baseline=(current bounding box.center)]
		\foreach \i in {0,1,2,3,4}{
			\draw[thick] (0,\i)--(4,\i);
			\draw[thick] (\i,0)--(\i,4);
		}
		\draw[ultra thick,->] (.9,.9)--(.1,.1);
		\draw[ultra thick,->] (1.9,1.9)--(1.1,1.1);
		\draw[ultra thick,->] (2.1,1.9)--(2.9,1.1);
		\draw[ultra thick,->] (3.1,.9)--(3.9,.1);
		\draw[ultra thick,->] (2.1,3.1)--(2.9,3.9);
		\draw[ultra thick,->] (3.1,2.1)--(3.9,2.9);
		\node[above left] at (1,0) {$a$};
		\node[above left] at (2,1) {$b$};
		\node[above right] at (2,1) {$c$};
		\node[above left] at (3,3) {$d$};
		\node[above left] at (4,2) {$f$};
		\node[above right] at (3,0) {$e$};
	\end{tikzpicture}
	\qquad
	\begin{tikzpicture}[scale=.8, baseline=(current bounding box.center)]
		\foreach \i in {0,1,2,3}{
			\draw[thick] (0,\i)--(3,\i);
			\draw[thick] (\i,0)--(\i,3);
		}
		\draw[ultra thick,->] (.9,.9)--(.1,.1);
		\draw[ultra thick,->] (1.1,2.1)--(1.9,2.9);
		\draw[ultra thick,->] (1.1,.9)--(1.9,.1);
		\draw[fill] (2.5,1.5) circle(.08);
		\node[above left] at (1,0) {$a$};
		\node[above right] at (1,0) {$b$};
		\node[above left] at (2,2) {$c$};
		\node[above left] at (3,1) {$d$};
	\end{tikzpicture}
\end{center}
\caption[The geometric grid classes of interest in this paper.]{The geometric grid classes of interest are
$\GG_1$, $\GG_2$, and $\GG_3$, whose standard figures are shown above from left to right. Note that $\GG_1$ also appears in \cite[Figure 11]{albert:inflations-case-studies}.}
\label{figure:our-classes}
\end{figure}

Let $\CC = \Geom(M)$ be a geometric grid class, and let $\Sigma$ be a set (an alphabet) that contains one letter for each nonempty cell of $M$, called the \emph{cell alphabet}. As usual with languages, let $\Sigma^*$ be the set of words formed by the alphabet $\Sigma$. Let $w  = w_1w_2\cdots w_n \in \Sigma^*$. We now describe how to map this word to a permutation in $\CC$.

Pick distances $0 < d_1 < d_2 < \cdots < d_n < 1$. The actual values of these distances does not matter, so long as no two are equal. For each letter $w_i$ in the word $w$, if the cell corresponding to $w_i$ has a directed line segment, then place a permutation entry in cell $w_i$ on the line segment at (infinity-norm) distance $d_i$ from the base point of that cell. If the cell corresponding to $w_i$ has just a single point, place the permutation entry at this point; there can be at most one entry per point. (Note that, as a cell with a $\bullet$ must lie in an otherwise empty row and column, we may ignore $d_i$ and put $w_i$ in the center of its cell.)

The result is a permutation drawn on the standard figure which we can read by labeling the entries in ascending order from bottom to top and then recording these labels from left to right.

Consider the leftmost geometric grid class in Figure~\ref{figure:our-classes}. The cell alphabet is $\Sigma = \{a,b,c,d\}$. Let \linebreak $w = bacddb \in \Sigma^*$. Then (using distances that are evenly spaced), the placement of entries onto the standard figure is as follows.
	
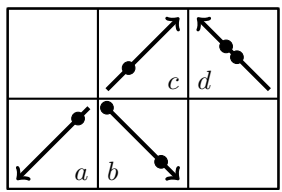
\begin{center}
		\begin{tikzpicture}[scale=1.2, baseline=(current bounding box.center)]
			\foreach \i in {0,1,2}{
				\draw[thick] (0,\i)--(3,\i);
				\draw[thick] (\i,0)--(\i,2);
			}
			\draw[thick] (0,0) rectangle (3,2);
			\draw[thick] (3,0)--(3,2);
			\draw[ultra thick,->] (.9,.9)--(.1,.1);
			\draw[ultra thick,<-] (1.9,1.9)--(1.1,1.1);
			\draw[ultra thick,<-] (2.1,1.9)--(2.9,1.1);
			\draw[ultra thick,->] (1.1,.9)--(1.9,.1);
			\node[above left] at (1,0) {$a$};
			\node[above right] at (1,0) {$b$};
			\node[above left] at (2,1) {$c$};
			\node[above right] at (2,1) {$d$};
			\draw[fill] (1.1,.9) circle (.07);
			\draw[fill] (.78,.78) circle (.07);
			\draw[fill] (1.34,1.34) circle (.07);
			\draw[fill] (2.54,1.46) circle (.07);
			\draw[fill] (2.42,1.58) circle (.07);
			\draw[fill] (1.7,.3) circle (.07);
		\end{tikzpicture}
	\end{center}
Numbering the entries in ascending order from bottom to top and then recording these entries from left to right gives the permutation $234165$. Given a geometric grid class $\CC$, the map described above is a surjection $\phi : \Sigma^* \to \CC$. 

Geometric grid classes are especially tractable because for any given geometric grid class $\CC$ we can construct a regular language $\LL$ such that there is a length-preserving bijection from $\LL$ to $\CC$. It is well-known that the generating function for the words in any regular language (by length) is rational. (See, for instance, Flajolet and Sedgewick~\cite[Section I.4 and Appendix A.7]{flajolet:ac}.) This fact is crucial in the proof of the following theorem.

\begin{theorem}[Albert, Atkinson, Bouvel, Ru\v{s}kuc, and Vatter~{\cite[Theorem 8.1]{albert:geometric-grid-classes}}]
	\label{theorem:rat}
	Every geometric grid class has a rational generating function.
\end{theorem}
	
Furthermore, this result extends to certain subsets of permutations in a geometrically grid class.

\begin{theorem}[Albert, Atkinson, Bouvel, Ru\v{s}kuc, and Vatter~{\cite[Theorem 9.1]{albert:geometric-grid-classes}}]
	The simple, sum indecomposable, and skew indecomposable permutations in every geometric grid class each have rational generating functions.
\end{theorem}

In~\cite{albert:geometric-grid-classes}, the authors describe the two obstacles which must be overcome in order to restrict the domain of $\phi$ to some regular language $\LL \subseteq \Sigma^*$ such that $\phi\res{\LL}$ is a bijection.\begin{description}
	\item{(1)} There are pairs of cells whose letters ``commute'', resulting in two words that map to the same permutation. These are the pairs that share neither a row nor a column. In the example above, since the cells $b$ and $d$ commute, the words $bacdbd$ and $bacbdd$ map to the same permutation. Such words have their entries in the same cells, but the entries are shifted around within each cell.
	\item{(2)} We can sometimes move entries between cells to produce the same word. In the example above, we can move the first three entries corresponding to $bac$ all into cell $a$ to yield the same permutation, i.e., $\phi(bacddb) = \phi(aaaddb)$. 
\end{description}

The first obstacle is easily dealt with: for each pair of commuting cells, we pick a preferred order, e.g., $bd$ instead of $db$, and we forbid all words that have any consecutive occurrence of a non-preferred order. The second issue is more delicate. In~\cite{albert:geometric-grid-classes}, a non-constructive proof is given to show that eliminating such duplicate words leaves a regular language. Since this proof does not lead to the construction of the regular language, we are required to find each of the regular languages by hand. Then, using the \texttt{Automata}~\cite{automata} package of \texttt{GAP}~\cite{gap}, we recover the rational generating function that counts each regular language. The code used to generate the results in this paper can be found on the author's website\footnote{At the time of publication, the author's website is located at \url{http://jaypantone.com}.}.

In Section \ref{section:simples} we show that the simple permutations of $\Av(3124,4312)$ are exactly the simple permutations in the union of $\GG_1$ and $\GG_2$. In Sections \ref{section:g1}, \ref{section:g2}, and \ref{section:g3} we construct the regular languages which are in bijection with all permutations and the simple permutations in each class $\GG_1$, $\GG_2$, and $\GG_3$, respectively. Then, we determine exactly how the simple permutations in each class may be inflated to yield permutations still in $\Av(3124,4312)$.

Section \ref{section:final-result} proves the following theorems about the rational generating function for $\Simples(\Av(3124,4312))$ and the algebraic generating function for $\Av(3124,4312)$.

\textbf{Theorem~\ref{theorem:simple-enum}}. The simple permutations in $\Av(3124,4312)$ are counted by the generating function
	\[
		S(x) =\frac{x - 2x^2 - 5x^3 + 12x^4 + x^5 - 8x^6 - 3x^7}{{\left(1 - 2x\right)} {\left(1 - x - x^2\right)}^{2}}.
	\]
	
\textbf{Theorem~\ref{theorem:class-enum}}. The permutations in $\Av(3124,4312)$ are counted by the generating function
	\[
		f(x) = \f{\pa{8x^5 - 16x^4 + 28x^3 - 26x^2 + 9x -1} + \sqrt{1-4x}\pa{2x^4-8x^3+14x^2-7x+1}}{2x^2(1-6x+9x^2-4x^3)}.
	\]

Lastly, in Section \ref{section:other-classes} we explore the applicability of this method to other permutation classes.

\section{The Simple Permutations of $\Av(3124,4312)$}\label{section:simples}

Consider the three geometric grid classes $\GG_1$, $\GG_2$, and $\GG_3$ whose standard figures are shown in Figure~\ref{figure:our-classes}. It is clear that $\GG_3$ is the intersection $\GG_1 \cap \GG_2$. This will be useful when we count the permutations in $\Av(3124,4312)$; those that arise as inflations of simple permutations that lie in both $\GG_1$ and $\GG_2$ will be double-counted, and hence in order to compensate, those that arise as inflations of simple permutations that lie in $\GG_3$ must be subtracted.

	
Whereas simple permutations of the classes examined by Albert, Atkinson, and Vatter~\cite{albert:inflations-case-studies} were each contained in a \emph{single} geometric grid class, the class studied here has simple permutations contained in the {union of two} geometric grid classes. This fact considerably lengthens both the proof of our upcoming Theorem~\ref{theorem:simples} and the subsequent analysis. Classes with this property have been enumerated before -- for example, by Albert, Atkinson, and Brignall~\cite{albert:enumeration-three-classes-grid} -- though the exact techniques have differed.

In the following arguments, we make frequent use of Albert's \texttt{PermLab} application~\cite{albert:permlab} to help us determine valid permutation configurations. We use permutation diagrams, comprised of a permutation plotted on top of a grid of cells. A cell is white if we are allowed to insert a new entry into that cell without creating an occurrence of $3124$ or $4312$. A cell is shaded dark gray if insertion into that cell would create a forbidden pattern, i.e., a $3124$ or $4312$. A cell is shaded light gray if we have specifically forbidden insertion into that cell as part of an argument, e.g., if we assume a particular entry is the maximal entry, then we can forbid insertion into all cells above it.

In order to talk about certain regions in a permutation diagram, we define the \emph{rectangular hull} of a set $S$ of points to be the smallest axis-parallel rectangle in the plane that contains all points of $S$. In particular, the rectangular hull of $S$ frequently contains additional points not in $S$. In our case, the points are entries in a permutation diagram. 

\begin{theorem}\label{theorem:simples} The simple permutations of $\Av(3124,4312)$ coincide with the simple permutations of $\GG_1 \cup \GG_2$. 
\end{theorem}
\begin{proof}
	It is clear by inspection that the permutations $3124$ and $4312$ cannot be drawn on the standard figures of $\GG_1$ or $\GG_2$. Therefore, $\Simples(\GG_1 \cup \GG_2) \subseteq \Simples(\Av(3124,4312))$. The reverse inclusion is much harder to show. Though the ideas in the proof are not particularly deep, we need to consider many cases.
	
	Let $\pi \in \Simples(\Av(3124,4312))$ have length $n$. 
	Let $\pi_L$ be the entries to the left of $n$ and let $\pi_R$ be the entries to the right of $n$, as in Figure~\ref{figure:split}. In order to avoid both $3124$ and $4312$ patterns, $\pi_L$ and $\pi_R$ must avoid $312$ patterns. Also, both $\pi_L$ and $\pi_R$ must be nonempty; otherwise $\pi$ begins or ends with $n$ and is not simple. 
	
	
	
	\begin{figure}
		\minipage{0.5\textwidth}
			\begin{center}
				\begin{tikzpicture}
					\draw[fill] (0,0) circle (.1) node[above=2pt] {$n$};
					\draw[ultra thick] (-1.5,-1.5) rectangle (-.2,-.2) node[midway] {$\pi_L$};
					\draw[ultra thick] (1.5,-1.5) rectangle (.2,-.2) node[midway] {$\pi_R$};
				\end{tikzpicture}
			\end{center}
			\caption{The decomposition of $\pi$ around the entry $n$ into $\pi_L$ and $\pi_R$.}
			\label{figure:split}
		\endminipage\hfill
		\minipage{0.5\textwidth}
			\begin{center}
				\begin{tikzpicture}[scale=1, baseline=(current bounding box.center)]
					\foreach \i in {0,1}{
						\draw[thick] (0,\i)--(1,\i);
						\draw[thick] (\i,0)--(\i,2);
					}
					\draw[thick] (0,2)--(1,2);
					\draw[ultra thick] (.1,.9)--(.9,.1);
					\draw[ultra thick] (.1,1.1)--(.9,1.9);
				\end{tikzpicture}
			\end{center}
			\caption{One class of wedge permutations, denoted $\Av(132,312)$.}
			\label{figure:class-parts}
		\endminipage
	\end{figure}
	
	\textbf{Case 1:} Assume $\pi_R$ does not contain the pattern $132$. Then, $\pi_R \in \Av(132,312)$ which implies that $\pi_R$ is a wedge permutation of the shape shown in Figure~\ref{figure:class-parts}. 
		
		\emph{Case 1a:} $\pi_R$ contains only the entry $\pi(n)$\\
		In order for $\pi$ to be simple, the rectangular hull of $n$ and $\pi(n)$ must be split to the left in the cell marked $A$ in Figure~\ref{figure:simples-case1-2}. We claim that any entries in cell $A$ must be increasing. To see this, assume there is a descent (a $21$ pattern) in cell $A$ and choose the `$2$' to be the topmost possible entry and the `$1$' to be the bottommost possible entry for the chosen $2$. This gives the diagram in Figure~\ref{figure:simples-case1-3}. The rectangular hull of the $21$ pattern formed by the leftmost two entries shown in Figure~\ref{figure:simples-case1-3} must be split to the left. Assume the separating entry is as far to the left as possible. Then, as seen in Figure~\ref{figure:simples-case1-4}, there exists an interval that cannot be split. Hence, any entries in cell $A$ must be increasing. The argument that we just made relating to the unsplittable $21$ pattern will be used many times in this proof. For its remainder, we will simply refer to an ``unsplittable $21$ pattern'' to mean that if we choose the `$2$' to be as high as possible and the `$1$' to be as low as possible, and then choose a separating entry as far to the left as possible, we get an interval of length 3 that cannot be split.
	
	In addition to being increasing, cell $A$ must also be nonempty. Consider an entry in cell $A$ that is as low as possible as possible, yielding Figure~\ref{figure:simples-case1-5}. Cell $B$ is empty because cell $A$ was assumed to be increasing. If $C$ has any descent, then there is an unsplittable $21$ pattern. If cell $D$ has an ascent, then there is a $3124$ patten. Hence, $B$ is empty, $C$ is increasing, and $D$ is decreasing. This leaves us with the diagram in Figure~\ref{figure:simples-case1-6}. It is clear that any permutation drawn on this figure lies in both $\GG_1$ and $\GG_2$. 
		
		\begin{figure}
		        \centering
		        \begin{subfigure}[b]{0.19\textwidth}
				\begin{center}
					\begin{tikzpicture}[scale=.4]
						\filldraw[light-gray](0,2) rectangle (1,3);
						\filldraw[light-gray](1,1) rectangle (2,2);
						\filldraw[light-gray](1,2) rectangle (2,3);
						\filldraw[light-gray](2,0) rectangle (3,1);
						\filldraw[light-gray](2,1) rectangle (3,2);
						\filldraw[light-gray](2,2) rectangle (3,3);
						\filldraw[light-gray](1,0) rectangle (2,1);
						\draw[black, fill=black] (1,2) circle (0.2);
						\draw[black, fill=black] (2,1) circle (0.2);
						\draw[thick](0,0)--(0,3);
						\draw[thick](0,0)--(3,0);
						\draw[thick](1,0)--(1,3);
						\draw[thick](0,1)--(3,1);
						\draw[thick](2,0)--(2,3);
						\draw[thick](0,2)--(3,2);
						\draw[thick](3,0)--(3,3);
						\draw[thick](0,3)--(3,3);
						\node at (.5,1.5) {$A$};
					\end{tikzpicture}
				\end{center}	
				\caption{}
				\label{figure:simples-case1-2}
		        \end{subfigure}%
		        \begin{subfigure}[b]{0.19\textwidth}
				\begin{center}
					\begin{tikzpicture}[scale=.4]
						\filldraw[light-gray](0,3) rectangle (1,4);
						\filldraw[light-gray](0,4) rectangle (1,5);
						\filldraw[dark-gray](1,0) rectangle (2,1);
						\filldraw[dark-gray](1,1) rectangle (2,2);
						\filldraw[light-gray](1,3) rectangle (2,4);
						\filldraw[light-gray](1,4) rectangle (2,5);
						\filldraw[dark-gray](2,0) rectangle (3,1);
						\filldraw[light-gray](2,1) rectangle (3,2);
						\filldraw[dark-gray](2,2) rectangle (3,3);
						\filldraw[light-gray](2,3) rectangle (3,4);
						\filldraw[light-gray](2,4) rectangle (3,5);
						\filldraw[dark-gray](3,0) rectangle (4,1);
						\filldraw[light-gray](3,1) rectangle (4,2);
						\filldraw[light-gray](3,2) rectangle (4,3);
						\filldraw[light-gray](3,3) rectangle (4,4);
						\filldraw[light-gray](3,4) rectangle (4,5);
						\filldraw[light-gray](4,0) rectangle (5,1);
						\filldraw[dark-gray](4,1) rectangle (5,2);
						\filldraw[light-gray](4,2) rectangle (5,3);
						\filldraw[light-gray](4,3) rectangle (5,4);
						\filldraw[light-gray](4,4) rectangle (5,5);
						\draw[black, fill=black] (1,3) circle (0.2);
						\draw[black, fill=black] (2,2) circle (0.2);
						\draw[black, fill=black] (3,4) circle (0.2);
						\draw[black, fill=black] (4,1) circle (0.2);
						\draw[thick](0,0)--(0,5);
						\draw[thick](0,0)--(5,0);
						\draw[thick](1,0)--(1,5);
						\draw[thick](0,1)--(5,1);
						\draw[thick](2,0)--(2,5);
						\draw[thick](0,2)--(5,2);
						\draw[thick](3,0)--(3,5);
						\draw[thick](0,3)--(5,3);
						\draw[thick](4,0)--(4,5);
						\draw[thick](0,4)--(5,4);
						\draw[thick](5,0)--(5,5);
						\draw[thick](0,5)--(5,5);
					\end{tikzpicture}
				\end{center}
				\caption{}
				\label{figure:simples-case1-3}
		        \end{subfigure}
		        \begin{subfigure}[b]{0.19\textwidth}
				\begin{center}
					\begin{tikzpicture}[scale=.4]
						\filldraw[light-gray](0,2) rectangle (1,3);
						\filldraw[light-gray](0,3) rectangle (1,4);
						\filldraw[dark-gray](0,4) rectangle (1,5);
						\filldraw[light-gray](0,5) rectangle (1,6);
						\filldraw[dark-gray](1,0) rectangle (2,1);
						\filldraw[dark-gray](1,1) rectangle (2,2);
						\filldraw[light-gray](1,4) rectangle (2,5);
						\filldraw[light-gray](1,5) rectangle (2,6);
						\filldraw[dark-gray](2,0) rectangle (3,1);
						\filldraw[dark-gray](2,1) rectangle (3,2);
						\filldraw[light-gray](2,4) rectangle (3,5);
						\filldraw[light-gray](2,5) rectangle (3,6);
						\filldraw[dark-gray](3,0) rectangle (4,1);
						\filldraw[dark-gray](3,2) rectangle (4,3);
						\filldraw[dark-gray](3,3) rectangle (4,4);
						\filldraw[light-gray](3,5) rectangle (4,6);
						\filldraw[dark-gray](4,0) rectangle (5,1);
						\filldraw[light-gray](4,1) rectangle (5,2);
						\filldraw[light-gray](4,2) rectangle (5,3);
						\filldraw[light-gray](4,3) rectangle (5,4);
						\filldraw[light-gray](4,4) rectangle (5,5);
						\filldraw[light-gray](4,5) rectangle (5,6);
						\filldraw[light-gray](5,0) rectangle (6,1);
						\filldraw[dark-gray](5,1) rectangle (6,2);
						\filldraw[light-gray](5,2) rectangle (6,3);
						\filldraw[light-gray](5,3) rectangle (6,4);
						\filldraw[light-gray](5,4) rectangle (6,5);
						\filldraw[light-gray](5,5) rectangle (6,6);
						\draw[black, fill=black] (1,3) circle (0.2);
						\draw[black, fill=black] (2,4) circle (0.2);
						\draw[black, fill=black] (3,2) circle (0.2);
						\draw[black, fill=black] (4,5) circle (0.2);
						\draw[black, fill=black] (5,1) circle (0.2);
						\draw[thick](0,0)--(0,6);
						\draw[thick](0,0)--(6,0);
						\draw[thick](1,0)--(1,6);
						\draw[thick](0,1)--(6,1);
						\draw[thick](2,0)--(2,6);
						\draw[thick](0,2)--(6,2);
						\draw[thick](3,0)--(3,6);
						\draw[thick](0,3)--(6,3);
						\draw[thick](4,0)--(4,6);
						\draw[thick](0,4)--(6,4);
						\draw[thick](5,0)--(5,6);
						\draw[thick](0,5)--(6,5);
						\draw[thick](6,0)--(6,6);
						\draw[thick](0,6)--(6,6);
					\end{tikzpicture}
				\end{center}
				\caption{}
				\label{figure:simples-case1-4}
		        \end{subfigure}
		          \begin{subfigure}[b]{0.19\textwidth}
				\begin{center}
					\begin{tikzpicture}[scale=.4]
						\filldraw[light-gray](0,1) rectangle (1,2);
						\filldraw[light-gray](0,3) rectangle (1,4);
						\filldraw[light-gray](1,1) rectangle (2,2);
						\filldraw[light-gray](1,3) rectangle (2,4);
						\filldraw[light-gray](2,0) rectangle (3,1);
						\filldraw[light-gray](2,1) rectangle (3,2);
						\filldraw[light-gray](2,2) rectangle (3,3);
						\filldraw[light-gray](2,3) rectangle (3,4);
						\filldraw[light-gray](3,0) rectangle (4,1);
						\filldraw[light-gray](3,1) rectangle (4,2);
						\filldraw[light-gray](3,2) rectangle (4,3);
						\filldraw[light-gray](3,3) rectangle (4,4);
						\draw[black, fill=black] (1,2) circle (0.2);
						\draw[black, fill=black] (2,3) circle (0.2);
						\draw[black, fill=black] (3,1) circle (0.2);
						\draw[thick](0,0)--(0,4);
						\draw[thick](0,0)--(4,0);
						\draw[thick](1,0)--(1,4);
						\draw[thick](0,1)--(4,1);
						\draw[thick](2,0)--(2,4);
						\draw[thick](0,2)--(4,2);
						\draw[thick](3,0)--(3,4);
						\draw[thick](0,3)--(4,3);
						\draw[thick](4,0)--(4,4);
						\draw[thick](0,4)--(4,4);
						\draw[ultra thick] (1.2,2.2)--(1.8,2.8);
						\node at (.5,2.5) {$B$};
						\node at (1.5,.5) {$D$};
						\node at (.5,.5) {$C$};
					\end{tikzpicture}
				\end{center}
				\caption{}
				\label{figure:simples-case1-5}
		        \end{subfigure}
		          \begin{subfigure}[b]{0.19\textwidth}
		                \begin{center}
					\begin{tikzpicture}[scale=.4]
						\filldraw[light-gray](0,1) rectangle (1,2);
						\filldraw[light-gray](0,3) rectangle (1,4);
						\filldraw[light-gray](1,1) rectangle (2,2);
						\filldraw[light-gray](1,3) rectangle (2,4);
						\filldraw[light-gray](2,0) rectangle (3,1);
						\filldraw[light-gray](2,1) rectangle (3,2);
						\filldraw[light-gray](2,2) rectangle (3,3);
						\filldraw[light-gray](2,3) rectangle (3,4);
						\filldraw[light-gray](3,0) rectangle (4,1);
						\filldraw[light-gray](3,1) rectangle (4,2);
						\filldraw[light-gray](3,2) rectangle (4,3);
						\filldraw[light-gray](3,3) rectangle (4,4);
						\filldraw[light-gray](0,2) rectangle (1,3);
						\draw[black, fill=black] (1,2) circle (0.2);
						\draw[black, fill=black] (2,3) circle (0.2);
						\draw[black, fill=black] (3,1) circle (0.2);
						\draw[thick](0,0)--(0,4);
						\draw[thick](0,0)--(4,0);
						\draw[thick](1,0)--(1,4);
						\draw[thick](0,1)--(4,1);
						\draw[thick](2,0)--(2,4);
						\draw[thick](0,2)--(4,2);
						\draw[thick](3,0)--(3,4);
						\draw[thick](0,3)--(4,3);
						\draw[thick](4,0)--(4,4);
						\draw[thick](0,4)--(4,4);
						\draw[ultra thick] (1.2,2.2)--(1.8,2.8);
						\draw[ultra thick] (1.2,0.8)--(1.8,0.2);
						\draw[ultra thick] (0.2,0.2)--(0.8,0.8);
					\end{tikzpicture}
				\end{center}
				\caption{}
				\label{figure:simples-case1-6}
		        \end{subfigure}

		\vspace*{.25cm}		        

		        \begin{subfigure}[b]{0.19\textwidth}
		              \begin{center}
					\begin{tikzpicture}[scale=.4]
						\filldraw[light-gray](0,3) rectangle (1,4);
						\filldraw[light-gray](1,0) rectangle (2,1);
						\filldraw[light-gray](1,1) rectangle (2,2);
						\filldraw[light-gray](1,2) rectangle (2,3);
						\filldraw[light-gray](1,3) rectangle (2,4);
						\filldraw[dark-gray](2,0) rectangle (3,1);
						\filldraw[light-gray](2,2) rectangle (3,3);
						\filldraw[light-gray](2,3) rectangle (3,4);
						\filldraw[light-gray](3,0) rectangle (4,1);
						\filldraw[dark-gray](3,1) rectangle (4,2);
						\filldraw[light-gray](3,2) rectangle (4,3);
						\filldraw[light-gray](3,3) rectangle (4,4);
						\draw[black, fill=black] (1,3) circle (0.2);
						\draw[black, fill=black] (2,2) circle (0.2);
						\draw[black, fill=black] (3,1) circle (0.2);
						\draw[thick](0,0)--(0,4);
						\draw[thick](0,0)--(4,0);
						\draw[thick](1,0)--(1,4);
						\draw[thick](0,1)--(4,1);
						\draw[thick](2,0)--(2,4);
						\draw[thick](0,2)--(4,2);
						\draw[thick](3,0)--(3,4);
						\draw[thick](0,3)--(4,3);
						\draw[thick](4,0)--(4,4);
						\draw[thick](0,4)--(4,4);
						\draw[ultra thick] (2.2,1.8)--(2.8,1.2);
					\end{tikzpicture}
				\end{center}
				\caption{}
				\label{figure:simples-new1}	
		        \end{subfigure}%
		        \begin{subfigure}[b]{0.19\textwidth}
				\begin{center}
					\begin{tikzpicture}[scale=.4]
						\filldraw[light-gray](0,1) rectangle (1,2);
						\filldraw[light-gray](0,2) rectangle (1,3);
						\filldraw[light-gray](0,4) rectangle (1,5);
						\filldraw[light-gray](1,4) rectangle (2,5);
						\filldraw[light-gray](2,0) rectangle (3,1);
						\filldraw[light-gray](2,1) rectangle (3,2);
						\filldraw[light-gray](2,2) rectangle (3,3);
						\filldraw[light-gray](2,3) rectangle (3,4);
						\filldraw[light-gray](2,4) rectangle (3,5);
						\filldraw[dark-gray](3,0) rectangle (4,1);
						\filldraw[light-gray](3,3) rectangle (4,4);
						\filldraw[light-gray](3,4) rectangle (4,5);
						\filldraw[light-gray](4,0) rectangle (5,1);
						\filldraw[dark-gray](4,1) rectangle (5,2);
						\filldraw[dark-gray](4,2) rectangle (5,3);
						\filldraw[light-gray](4,3) rectangle (5,4);
						\filldraw[light-gray](4,4) rectangle (5,5);
						\draw[black, fill=black] (1,2) circle (0.2);
						\draw[black, fill=black] (2,4) circle (0.2);
						\draw[black, fill=black] (3,3) circle (0.2);
						\draw[black, fill=black] (4,1) circle (0.2);
						\draw[thick](0,0)--(0,5);
						\draw[thick](0,0)--(5,0);
						\draw[thick](1,0)--(1,5);
						\draw[thick](0,1)--(5,1);
						\draw[thick](2,0)--(2,5);
						\draw[thick](0,2)--(5,2);
						\draw[thick](3,0)--(3,5);
						\draw[thick](0,3)--(5,3);
						\draw[thick](4,0)--(4,5);
						\draw[thick](0,4)--(5,4);
						\draw[thick](5,0)--(5,5);
						\draw[thick](0,5)--(5,5);
						\draw[ultra thick] (3.2,2.8)--(3.8,1.2);
						\node at (.5,3.5) {$A$};
						\node at (1.5,.5) {$C$};
						\node at (1.5,1.5) {$B$};
						\node at (.5,.5) {$D$};
						\node at (1.5,3.5) {$E$};
					\end{tikzpicture}
				\end{center}	
				\caption{}
				\label{figure:simples-new2}
		        \end{subfigure}
		        \begin{subfigure}[b]{0.19\textwidth}
				\begin{center}
					\begin{tikzpicture}[scale=.4]
						\filldraw[light-gray](0,1) rectangle (1,2);
						\filldraw[light-gray](0,2) rectangle (1,3);
						\filldraw[light-gray](0,4) rectangle (1,5);
						\filldraw[light-gray](1,4) rectangle (2,5);
						\filldraw[light-gray](2,0) rectangle (3,1);
						\filldraw[light-gray](2,1) rectangle (3,2);
						\filldraw[light-gray](2,2) rectangle (3,3);
						\filldraw[light-gray](2,3) rectangle (3,4);
						\filldraw[light-gray](2,4) rectangle (3,5);
						\filldraw[dark-gray](3,0) rectangle (4,1);
						\filldraw[light-gray](3,3) rectangle (4,4);
						\filldraw[light-gray](3,4) rectangle (4,5);
						\filldraw[light-gray](4,0) rectangle (5,1);
						\filldraw[dark-gray](4,1) rectangle (5,2);
						\filldraw[dark-gray](4,2) rectangle (5,3);
						\filldraw[light-gray](4,3) rectangle (5,4);
						\filldraw[light-gray](4,4) rectangle (5,5);
						\filldraw[light-gray](0,3) rectangle (1,4);
						\draw[black, fill=black] (1,2) circle (0.2);
						\draw[black, fill=black] (2,4) circle (0.2);
						\draw[black, fill=black] (3,3) circle (0.2);
						\draw[black, fill=black] (4,1) circle (0.2);
						\draw[thick](0,0)--(0,5);
						\draw[thick](0,0)--(5,0);
						\draw[thick](1,0)--(1,5);
						\draw[thick](0,1)--(5,1);
						\draw[thick](2,0)--(2,5);
						\draw[thick](0,2)--(5,2);
						\draw[thick](3,0)--(3,5);
						\draw[thick](0,3)--(5,3);
						\draw[thick](4,0)--(4,5);
						\draw[thick](0,4)--(5,4);
						\draw[thick](5,0)--(5,5);
						\draw[thick](0,5)--(5,5);
						\draw[ultra thick] (3.2,2.8)--(3.8,1.2);
						\draw[ultra thick] (1.2,1.8)--(1.8,0.2);
						\draw[ultra thick] (0.2,0.2)--(0.8,0.8);
						\draw[ultra thick] (1.2,3.2)--(1.8,3.8);
					\end{tikzpicture}
				\end{center}	
				\caption{}
				\label{figure:simples-new3}
		        \end{subfigure}
		          \begin{subfigure}[b]{0.19\textwidth}
				\begin{center}
					\begin{tikzpicture}[scale=.4]
						\filldraw[light-gray](0,1) rectangle (1,2);
						\filldraw[light-gray](0,2) rectangle (1,3);
						\filldraw[light-gray](0,3) rectangle (1,4);
						\filldraw[dark-gray](0,4) rectangle (1,5);
						\filldraw[light-gray](0,5) rectangle (1,6);
						\filldraw[light-gray](1,3) rectangle (2,4);
						\filldraw[light-gray](1,4) rectangle (2,5);
						\filldraw[light-gray](1,5) rectangle (2,6);
						\filldraw[light-gray](2,3) rectangle (3,4);
						\filldraw[light-gray](2,5) rectangle (3,6);
						\filldraw[light-gray](3,0) rectangle (4,1);
						\filldraw[light-gray](3,1) rectangle (4,2);
						\filldraw[light-gray](3,2) rectangle (4,3);
						\filldraw[light-gray](3,3) rectangle (4,4);
						\filldraw[light-gray](3,4) rectangle (4,5);
						\filldraw[light-gray](3,5) rectangle (4,6);
						\filldraw[dark-gray](4,0) rectangle (5,1);
						\filldraw[light-gray](4,3) rectangle (5,4);
						\filldraw[light-gray](4,4) rectangle (5,5);
						\filldraw[light-gray](4,5) rectangle (5,6);
						\filldraw[light-gray](5,0) rectangle (6,1);
						\filldraw[dark-gray](5,1) rectangle (6,2);
						\filldraw[dark-gray](5,2) rectangle (6,3);
						\filldraw[light-gray](5,3) rectangle (6,4);
						\filldraw[light-gray](5,4) rectangle (6,5);
						\filldraw[light-gray](5,5) rectangle (6,6);
						\draw[black, fill=black] (1,2) circle (0.2);
						\draw[black, fill=black] (2,4) circle (0.2);
						\draw[black, fill=black] (3,5) circle (0.2);
						\draw[black, fill=black] (4,3) circle (0.2);
						\draw[black, fill=black] (5,1) circle (0.2);
						\draw[thick](0,0)--(0,6);
						\draw[thick](0,0)--(6,0);
						\draw[thick](1,0)--(1,6);
						\draw[thick](0,1)--(6,1);
						\draw[thick](2,0)--(2,6);
						\draw[thick](0,2)--(6,2);
						\draw[thick](3,0)--(3,6);
						\draw[thick](0,3)--(6,3);
						\draw[thick](4,0)--(4,6);
						\draw[thick](0,4)--(6,4);
						\draw[thick](5,0)--(5,6);
						\draw[thick](0,5)--(6,5);
						\draw[thick](6,0)--(6,6);
						\draw[thick](0,6)--(6,6);
						\node at (1.5,0.5) {$F$};
						\node at (1.5,1.5) {$G$};
						\node at (2.5,1.5) {$H$};
						\node at (2.5,0.5) {$I$};
						\node at (1.5,2.5) {$J$};
						\node at (2.5,2.5) {$K$};
						\draw[ultra thick] (4.2,2.8)--(4.8,1.2);
						\draw[ultra thick] (0.2,0.2)--(0.8,0.8);
						\draw[ultra thick] (2.2,4.2)--(2.8,4.8);
					\end{tikzpicture}
				\end{center}	
				\caption{}
				\label{figure:simples-new4}
		        \end{subfigure}
		          \begin{subfigure}[b]{0.19\textwidth}
				\begin{center}
\begin{tikzpicture}[scale=.4]
\filldraw[light-gray](0,1) rectangle (1,2);
\filldraw[light-gray](0,2) rectangle (1,3);
\filldraw[light-gray](0,3) rectangle (1,4);
\filldraw[light-gray](0,4) rectangle (1,5);
\filldraw[dark-gray](0,5) rectangle (1,6);
\filldraw[dark-gray](0,6) rectangle (1,7);
\filldraw[dark-gray](1,0) rectangle (2,1);
\filldraw[dark-gray](1,1) rectangle (2,2);
\filldraw[light-gray](1,4) rectangle (2,5);
\filldraw[dark-gray](1,5) rectangle (2,6);
\filldraw[dark-gray](1,6) rectangle (2,7);
\filldraw[dark-gray](2,0) rectangle (3,1);
\filldraw[dark-gray](2,1) rectangle (3,2);
\filldraw[light-gray](2,3) rectangle (3,4);
\filldraw[dark-gray](2,4) rectangle (3,5);
\filldraw[light-gray](2,6) rectangle (3,7);
\filldraw[dark-gray](3,0) rectangle (4,1);
\filldraw[dark-gray](3,2) rectangle (4,3);
\filldraw[dark-gray](3,3) rectangle (4,4);
\filldraw[dark-gray](3,4) rectangle (4,5);
\filldraw[light-gray](3,6) rectangle (4,7);
\filldraw[dark-gray](4,0) rectangle (5,1);
\filldraw[light-gray](4,1) rectangle (5,2);
\filldraw[dark-gray](4,2) rectangle (5,3);
\filldraw[light-gray](4,3) rectangle (5,4);
\filldraw[light-gray](4,4) rectangle (5,5);
\filldraw[light-gray](4,5) rectangle (5,6);
\filldraw[light-gray](4,6) rectangle (5,7);
\filldraw[dark-gray](5,0) rectangle (6,1);
\filldraw[light-gray](5,4) rectangle (6,5);
\filldraw[dark-gray](5,5) rectangle (6,6);
\filldraw[dark-gray](5,6) rectangle (6,7);
\filldraw[light-gray](6,0) rectangle (7,1);
\filldraw[dark-gray](6,1) rectangle (7,2);
\filldraw[dark-gray](6,2) rectangle (7,3);
\filldraw[dark-gray](6,3) rectangle (7,4);
\filldraw[light-gray](6,4) rectangle (7,5);
\filldraw[dark-gray](6,5) rectangle (7,6);
\filldraw[dark-gray](6,6) rectangle (7,7);
\draw[black, fill=black] (1,3) circle (0.2);
\draw[black, fill=black] (2,5) circle (0.2);
\draw[black, fill=black] (3,2) circle (0.2);
\draw[black, fill=black] (4,6) circle (0.2);
\draw[black, fill=black] (5,4) circle (0.2);
\draw[black, fill=black] (6,1) circle (0.2);
\draw[thick](0,0)--(0,7);
\draw[thick](0,0)--(7,0);
\draw[thick](1,0)--(1,7);
\draw[thick](0,1)--(7,1);
\draw[thick](2,0)--(2,7);
\draw[thick](0,2)--(7,2);
\draw[thick](3,0)--(3,7);
\draw[thick](0,3)--(7,3);
\draw[thick](4,0)--(4,7);
\draw[thick](0,4)--(7,4);
\draw[thick](5,0)--(5,7);
\draw[thick](0,5)--(7,5);
\draw[thick](6,0)--(6,7);
\draw[thick](0,6)--(7,6);
\draw[thick](7,0)--(7,7);
\draw[thick](0,7)--(7,7);

\draw[ultra thick] (2.2, 5.2) -- (3.8, 5.8);
\draw[ultra thick] (5.2, 3.8) -- (5.8, 1.2);
\draw[ultra thick] (.2, .2) -- (.8, .8);
\end{tikzpicture}

				\end{center}	
				\caption{}
				\label{figure:simples-new5}
		        \end{subfigure}
		        \caption{Permutation diagrams corresponding to steps in the proof of Theorem \ref{theorem:simples}.}
		\end{figure}
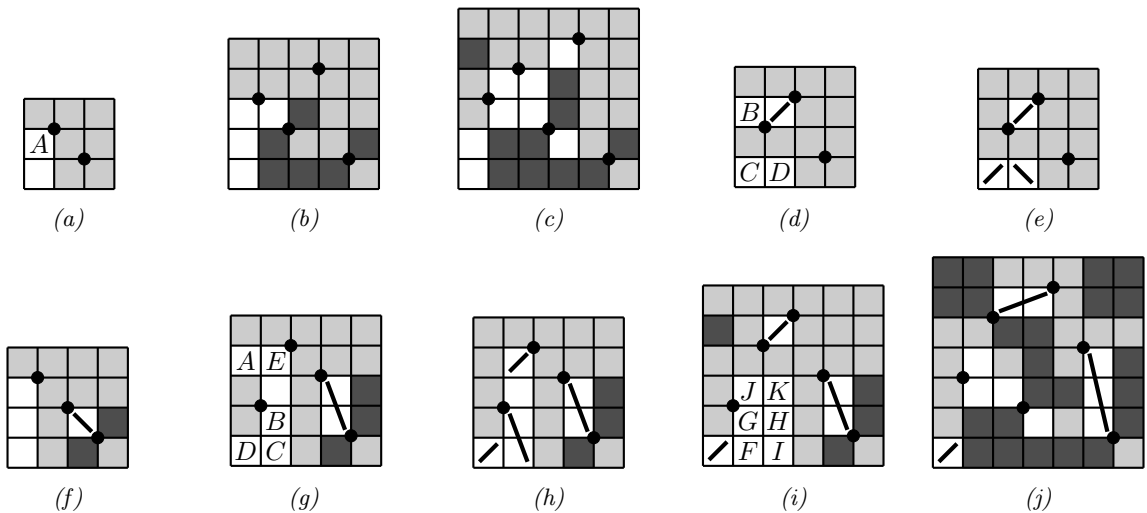

		\emph{Case 1b:} $\pi_R$ has more than one entry, but is strictly decreasing\\
		In this case, we have the diagram in Figure~\ref{figure:simples-new1}. The rectangular hull of $\pi_R$ must be split with an entry to the left. Choose the leftmost possible entry to get the diagram in Figure~\ref{figure:simples-new2}. If cell $A$ contains an entry, then the interval contained in the rectangular hull of all entries in cell $A$ together with the leftmost three entries shown in Figure~\ref{figure:simples-new2} cannot be split. Hence, cell $A$ is empty. Cells $B$ and $C$ together must be decreasing to avoid a $3124$ pattern. To avoid an unsplittable interval contained entirely with cell $D$, cell $D$ must be increasing. For the same reason, cell $E$ must also be increasing. Thus, we have the situation in Figure~\ref{figure:simples-new3}.
		
		Now, the rectangular hull of $n$ and the entry to its immediate right must be split to the left. The result is Figure~\ref{figure:simples-new4}. To avoid a $3124$ pattern, cells $F$, $G$, $H$, and $I$ (which came from cells $B$ and $C$) must together form one decreasing block. If cell $K$ is nonempty, then there is an unsplittable interval that contains all entries in cell $K$, the entry $n$, the entry just to the right of $n$, the monotone increasing cell just below and to the left of $n$, and possibly (if cell $F$ is empty) the leftmost entry shown in Figure~\ref{figure:simples-new4} along with any entries in cells $G$, $H$, and $J$. Hence, cell $K$ is empty. Suppose that cell $H$ has an entry. Then, $\pi$ is represented by Figure~\ref{figure:simples-new5}, and it is clear that any permutation drawn on this diagram must have the form $1 \oplus \tau$ or $\tau \ominus 1$ (for some $\tau$) and thus cannot be simple. If cell $G$ has an entry, then an almost identical situation occurs. Therefore, $\pi$ has the form shown in Figure~\ref{figure:simples-new55}. 
		
		If $J$ contains a descent, then (as can be seen in Figure~\ref{figure:simples-new56}), the monotone increasing interval containing $n$ cannot be split. Thus $J$ must be increasing, as shown in Figure~\ref{figure:simples-new57}, and is now clear that $\pi \in \Simples(\GG_1)$.

			
		\emph{Case 1c:} $\pi_R$ is not strictly decreasing\\
		First we rule out a form of $\pi$ that cannot occur. We know that $\pi_R$ has the shape of a wedge permutation as shown in Figure~\ref{figure:class-parts}. We now show that if $\pi_R$ is not strictly decreasing then its rightmost entry, $\pi(n)$, lies on the top part of the wedge permutation, not the bottom part. Assume otherwise. Then, we have the picture in Figure~\ref{figure:wedge2}, where the three rightmost entries shown are (from left to right) the leftmost entry in the wedge, the topmost entry in the wedge, and the rightmost (and bottommost) entry in the wedge.
				
		If cell $L$ of Figure~\ref{figure:wedge2} were empty, then we would have $\pi(n)=1$, which contradicts the simplicity of $\pi$. As such, cell $L$ must be nonempty, yielding Figure~\ref{figure:wedge25}.  The entry in $\pi$ following the leftmost entry of $\pi_R$ (the second leftmost entry of the wedge) is either above or below the leftmost entry of $\pi_R$ (note that we are not talking about the second leftmost entry of $\pi_R$ shown on the diagram, but the second leftmost entry of $\pi_R$ in any permutation drawn on the diagram, i.e., this entry may, and probably will, lie in one of the white cells). We describe the first case in detail; the second case follows by an almost identical argument. Figure~\ref{figure:wedge26} shows $\pi$ under the assumption that the second leftmost entry in the wedge lies above the leftmost entry in the wedge. The two leftmost entries of $\pi_R$ must be separated by an entry in the white cell to their left; suppose this splitting entry is as leftmost as possible, as shown in Figure~\ref{figure:wedge27}. We can now show that the uppermost five entries shown in Figure~\ref{figure:wedge27} are part of an interval that cannot be split. There are two open cells in which a splitting entry could lie: cells $A$ and $B$. Obviously, if both cells $A$ and $B$ are empty, then the interval in question cannot be split. In fact, $B$ must contain at least one entry in order to split the monotone increasing interval above it. Suppose the entry in $B$ is drawn as bottommost as possible, as in Figure~\ref{figure:wedge28}. If cell $C$ has an entry, then the rectangular hull of its leftmost entry together with all entries shown in Figure~\ref{figure:wedge28} except the leftmost and rightmost is an interval. Since cell $C$ is thus empty, cell $A$ must contain an entry. Placing an entry in cell $A$ has a similar effect as the entry that we previously placed in cell $C$, and thus we reach a contradiction to the assumption that the rightmost entry in $\pi_R$ lies on the bottom leg of the wedge.
		
	\begin{figure}
		        \centering
		        \begin{subfigure}[b]{0.24\textwidth}
				\begin{center}
					\begin{tikzpicture}[scale=.4]
						\filldraw[light-gray](0,1) rectangle (1,2);
						\filldraw[light-gray](0,2) rectangle (1,3);
						\filldraw[light-gray](0,3) rectangle (1,4);
						\filldraw[dark-gray](0,4) rectangle (1,5);
						\filldraw[light-gray](0,5) rectangle (1,6);
						\filldraw[light-gray](1,3) rectangle (2,4);
						\filldraw[light-gray](1,4) rectangle (2,5);
						\filldraw[light-gray](1,5) rectangle (2,6);
						\filldraw[light-gray](2,3) rectangle (3,4);
						\filldraw[light-gray](2,5) rectangle (3,6);
						\filldraw[light-gray](3,0) rectangle (4,1);
						\filldraw[light-gray](3,1) rectangle (4,2);
						\filldraw[light-gray](3,2) rectangle (4,3);
						\filldraw[light-gray](3,3) rectangle (4,4);
						\filldraw[light-gray](3,4) rectangle (4,5);
						\filldraw[light-gray](3,5) rectangle (4,6);
						\filldraw[dark-gray](4,0) rectangle (5,1);
						\filldraw[light-gray](4,3) rectangle (5,4);
						\filldraw[light-gray](4,4) rectangle (5,5);
						\filldraw[light-gray](4,5) rectangle (5,6);
						\filldraw[light-gray](5,0) rectangle (6,1);
						\filldraw[dark-gray](5,1) rectangle (6,2);
						\filldraw[dark-gray](5,2) rectangle (6,3);
						\filldraw[light-gray](5,3) rectangle (6,4);
						\filldraw[light-gray](5,4) rectangle (6,5);
						\filldraw[light-gray](5,5) rectangle (6,6);
						\filldraw[light-gray](2,2) rectangle (3,3);
						\filldraw[light-gray](2,1) rectangle (3,2);
						\filldraw[light-gray](1,1) rectangle (2,2);
						\draw[black, fill=black] (1,2) circle (0.2);
						\draw[black, fill=black] (2,4) circle (0.2);
						\draw[black, fill=black] (3,5) circle (0.2);
						\draw[black, fill=black] (4,3) circle (0.2);
						\draw[black, fill=black] (5,1) circle (0.2);
						\draw[thick](0,0)--(0,6);
						\draw[thick](0,0)--(6,0);
						\draw[thick](1,0)--(1,6);
						\draw[thick](0,1)--(6,1);
						\draw[thick](2,0)--(2,6);
						\draw[thick](0,2)--(6,2);
						\draw[thick](3,0)--(3,6);
						\draw[thick](0,3)--(6,3);
						\draw[thick](4,0)--(4,6);
						\draw[thick](0,4)--(6,4);
						\draw[thick](5,0)--(5,6);
						\draw[thick](0,5)--(6,5);
						\draw[thick](6,0)--(6,6);
						\draw[thick](0,6)--(6,6);
						\node at (1.5,2.5) {$J$};
						\draw[ultra thick] (4.2,2.8)--(4.8,1.2);
						\draw[ultra thick] (0.2,0.2)--(0.8,0.8);
						\draw[ultra thick] (2.2,4.2)--(2.8,4.8);
						\draw[ultra thick] (1.2,0.8)--(2.8, 0.2);
					\end{tikzpicture}
				\end{center}	
				\caption{}
				\label{figure:simples-new55}
		        \end{subfigure}
		         \begin{subfigure}[b]{0.24\textwidth}
				\begin{center}
\begin{tikzpicture}[scale=.4]
\filldraw[light-gray](0,1) rectangle (1,2);
\filldraw[light-gray](0,2) rectangle (1,3);
\filldraw[dark-gray](0,3) rectangle (1,4);
\filldraw[dark-gray](0,4) rectangle (1,5);
\filldraw[dark-gray](0,5) rectangle (1,6);
\filldraw[dark-gray](0,6) rectangle (1,7);
\filldraw[light-gray](0,7) rectangle (1,8);
\filldraw[light-gray](1,1) rectangle (2,2);
\filldraw[light-gray](1,5) rectangle (2,6);
\filldraw[dark-gray](1,6) rectangle (2,7);
\filldraw[light-gray](1,7) rectangle (2,8);
\filldraw[dark-gray](2,0) rectangle (3,1);
\filldraw[dark-gray](2,1) rectangle (3,2);
\filldraw[dark-gray](2,2) rectangle (3,3);
\filldraw[light-gray](2,5) rectangle (3,6);
\filldraw[dark-gray](2,6) rectangle (3,7);
\filldraw[light-gray](2,7) rectangle (3,8);
\filldraw[dark-gray](3,0) rectangle (4,1);
\filldraw[light-gray](3,1) rectangle (4,2);
\filldraw[dark-gray](3,3) rectangle (4,4);
\filldraw[light-gray](3,5) rectangle (4,6);
\filldraw[light-gray](3,6) rectangle (4,7);
\filldraw[light-gray](3,7) rectangle (4,8);
\filldraw[dark-gray](4,0) rectangle (5,1);
\filldraw[light-gray](4,1) rectangle (5,2);
\filldraw[light-gray](4,2) rectangle (5,3);
\filldraw[dark-gray](4,3) rectangle (5,4);
\filldraw[light-gray](4,4) rectangle (5,5);
\filldraw[light-gray](4,5) rectangle (5,6);
\filldraw[light-gray](4,7) rectangle (5,8);
\filldraw[dark-gray](5,0) rectangle (6,1);
\filldraw[light-gray](5,1) rectangle (6,2);
\filldraw[light-gray](5,2) rectangle (6,3);
\filldraw[dark-gray](5,3) rectangle (6,4);
\filldraw[light-gray](5,4) rectangle (6,5);
\filldraw[light-gray](5,5) rectangle (6,6);
\filldraw[light-gray](5,6) rectangle (6,7);
\filldraw[light-gray](5,7) rectangle (6,8);
\filldraw[dark-gray](6,0) rectangle (7,1);
\filldraw[light-gray](6,5) rectangle (7,6);
\filldraw[light-gray](6,6) rectangle (7,7);
\filldraw[light-gray](6,7) rectangle (7,8);
\filldraw[light-gray](7,0) rectangle (8,1);
\filldraw[dark-gray](7,1) rectangle (8,2);
\filldraw[dark-gray](7,2) rectangle (8,3);
\filldraw[dark-gray](7,3) rectangle (8,4);
\filldraw[dark-gray](7,4) rectangle (8,5);
\filldraw[light-gray](7,5) rectangle (8,6);
\filldraw[light-gray](7,6) rectangle (8,7);
\filldraw[light-gray](7,7) rectangle (8,8);
\draw[black, fill=black] (1,2) circle (0.2);
\draw[black, fill=black] (2,4) circle (0.2);
\draw[black, fill=black] (3,3) circle (0.2);
\draw[black, fill=black] (4,6) circle (0.2);
\draw[black, fill=black] (5,7) circle (0.2);
\draw[black, fill=black] (6,5) circle (0.2);
\draw[black, fill=black] (7,1) circle (0.2);
\draw[thick](0,0)--(0,8);
\draw[thick](0,0)--(8,0);
\draw[thick](1,0)--(1,8);
\draw[thick](0,1)--(8,1);
\draw[thick](2,0)--(2,8);
\draw[thick](0,2)--(8,2);
\draw[thick](3,0)--(3,8);
\draw[thick](0,3)--(8,3);
\draw[thick](4,0)--(4,8);
\draw[thick](0,4)--(8,4);
\draw[thick](5,0)--(5,8);
\draw[thick](0,5)--(8,5);
\draw[thick](6,0)--(6,8);
\draw[thick](0,6)--(8,6);
\draw[thick](7,0)--(7,8);
\draw[thick](0,7)--(8,7);
\draw[thick](8,0)--(8,8);
\draw[thick](0,8)--(8,8);
						\draw[ultra thick] (6.2,4.8)--(6.8,1.2);
						\draw[ultra thick] (0.2,0.2)--(0.8,0.8);
						\draw[ultra thick] (4.2,6.2)--(4.8,6.8);
						\draw[ultra thick] (1.2,0.8)--(1.8, 0.2);
\end{tikzpicture}
				\end{center}	
				\caption{}
				\label{figure:simples-new56}
		        \end{subfigure}
		         \begin{subfigure}[b]{0.24\textwidth}
				\begin{center}
					\begin{tikzpicture}[scale=.4]
						\filldraw[light-gray](0,1) rectangle (1,2);
						\filldraw[light-gray](0,2) rectangle (1,3);
						\filldraw[light-gray](0,3) rectangle (1,4);
						\filldraw[dark-gray](0,4) rectangle (1,5);
						\filldraw[light-gray](0,5) rectangle (1,6);
						\filldraw[light-gray](1,3) rectangle (2,4);
						\filldraw[light-gray](1,4) rectangle (2,5);
						\filldraw[light-gray](1,5) rectangle (2,6);
						\filldraw[light-gray](2,3) rectangle (3,4);
						\filldraw[light-gray](2,5) rectangle (3,6);
						\filldraw[light-gray](3,0) rectangle (4,1);
						\filldraw[light-gray](3,1) rectangle (4,2);
						\filldraw[light-gray](3,2) rectangle (4,3);
						\filldraw[light-gray](3,3) rectangle (4,4);
						\filldraw[light-gray](3,4) rectangle (4,5);
						\filldraw[light-gray](3,5) rectangle (4,6);
						\filldraw[dark-gray](4,0) rectangle (5,1);
						\filldraw[light-gray](4,3) rectangle (5,4);
						\filldraw[light-gray](4,4) rectangle (5,5);
						\filldraw[light-gray](4,5) rectangle (5,6);
						\filldraw[light-gray](5,0) rectangle (6,1);
						\filldraw[dark-gray](5,1) rectangle (6,2);
						\filldraw[dark-gray](5,2) rectangle (6,3);
						\filldraw[light-gray](5,3) rectangle (6,4);
						\filldraw[light-gray](5,4) rectangle (6,5);
						\filldraw[light-gray](5,5) rectangle (6,6);
						\filldraw[light-gray](2,2) rectangle (3,3);
						\filldraw[light-gray](2,1) rectangle (3,2);
						\filldraw[light-gray](1,1) rectangle (2,2);
						\draw[black, fill=black] (1,2) circle (0.2);
						\draw[black, fill=black] (2,4) circle (0.2);
						\draw[black, fill=black] (3,5) circle (0.2);
						\draw[black, fill=black] (4,3) circle (0.2);
						\draw[black, fill=black] (5,1) circle (0.2);
						\draw[thick](0,0)--(0,6);
						\draw[thick](0,0)--(6,0);
						\draw[thick](1,0)--(1,6);
						\draw[thick](0,1)--(6,1);
						\draw[thick](2,0)--(2,6);
						\draw[thick](0,2)--(6,2);
						\draw[thick](3,0)--(3,6);
						\draw[thick](0,3)--(6,3);
						\draw[thick](4,0)--(4,6);
						\draw[thick](0,4)--(6,4);
						\draw[thick](5,0)--(5,6);
						\draw[thick](0,5)--(6,5);
						\draw[thick](6,0)--(6,6);
						\draw[thick](0,6)--(6,6);
						\draw[ultra thick] (1.2,2.2)--(1.8,2.8);
						\draw[ultra thick] (4.2,2.8)--(4.8,1.2);
						\draw[ultra thick] (0.2,0.2)--(0.8,0.8);
						\draw[ultra thick] (2.2,4.2)--(2.8,4.8);
						\draw[ultra thick] (1.2,0.8)--(2.8, 0.2);
					\end{tikzpicture}
				\end{center}
				\caption{}
				\label{figure:simples-new57}
		        \end{subfigure}
		        \begin{subfigure}[b]{0.24\textwidth}
				\begin{center}
					\begin{tikzpicture}[scale=.4]
						\filldraw[dark-gray](0,4) rectangle (1,5);
						\filldraw[light-gray](1,0) rectangle (2,1);
						\filldraw[light-gray](1,1) rectangle (2,2);
						\filldraw[light-gray](1,2) rectangle (2,3);
						\filldraw[dark-gray](1,3) rectangle (2,4);
						\filldraw[light-gray](1,4) rectangle (2,5);
						\filldraw[dark-gray](2,0) rectangle (3,1);
						\filldraw[light-gray](2,3) rectangle (3,4);
						\filldraw[light-gray](2,4) rectangle (3,5);
						\filldraw[dark-gray](3,0) rectangle (4,1);
						\filldraw[light-gray](3,3) rectangle (4,4);
						\filldraw[dark-gray](3,4) rectangle (4,5);
						\filldraw[light-gray](4,0) rectangle (5,1);
						\filldraw[dark-gray](4,1) rectangle (5,2);
						\filldraw[dark-gray](4,2) rectangle (5,3);
						\filldraw[light-gray](4,3) rectangle (5,4);
						\filldraw[dark-gray](4,4) rectangle (5,5);
						\filldraw[light-gray](3,2) rectangle (4,3);
						\draw[black, fill=black] (1,4) circle (0.2);
						\draw[black, fill=black] (2,2) circle (0.2);
						\draw[black, fill=black] (3,3) circle (0.2);
						\draw[black, fill=black] (4,1) circle (0.2);
						\draw[thick](0,0)--(0,5);
						\draw[thick](0,0)--(5,0);
						\draw[thick](1,0)--(1,5);
						\draw[thick](0,1)--(5,1);
						\draw[thick](2,0)--(2,5);
						\draw[thick](0,2)--(5,2);
						\draw[thick](3,0)--(3,5);
						\draw[thick](0,3)--(5,3);
						\draw[thick](4,0)--(4,5);
						\draw[thick](0,4)--(5,4);
						\draw[thick](5,0)--(5,5);
						\draw[thick](0,5)--(5,5);
						\draw[ultra thick] (2.2,1.8)--(3.8,1.2);
						\draw[ultra thick] (2.2, 2.2)--(2.8, 2.8);
						\node at (.5,.5) {$L$};
					\end{tikzpicture}
				\end{center}
				\caption{}
				\label{figure:wedge2}

		        \end{subfigure}
		        
		\vspace*{.25cm}
		
		\begin{subfigure}[b]{0.24\textwidth}
				\begin{center}
\begin{tikzpicture}[scale=.4]
\filldraw[light-gray](0,0) rectangle (1,1);
\filldraw[dark-gray](0,3) rectangle (1,4);
\filldraw[dark-gray](0,5) rectangle (1,6);
\filldraw[light-gray](1,0) rectangle (2,1);
\filldraw[dark-gray](1,5) rectangle (2,6);
\filldraw[light-gray](2,0) rectangle (3,1);
\filldraw[light-gray](2,1) rectangle (3,2);
\filldraw[light-gray](2,2) rectangle (3,3);
\filldraw[light-gray](2,3) rectangle (3,4);
\filldraw[dark-gray](2,4) rectangle (3,5);
\filldraw[light-gray](2,5) rectangle (3,6);
\filldraw[dark-gray](3,0) rectangle (4,1);
\filldraw[dark-gray](3,1) rectangle (4,2);
\filldraw[light-gray](3,4) rectangle (4,5);
\filldraw[light-gray](3,5) rectangle (4,6);
\filldraw[dark-gray](4,0) rectangle (5,1);
\filldraw[dark-gray](4,1) rectangle (5,2);
\filldraw[light-gray](4,3) rectangle (5,4);
\filldraw[light-gray](4,4) rectangle (5,5);
\filldraw[dark-gray](4,5) rectangle (5,6);
\filldraw[light-gray](5,0) rectangle (6,1);
\filldraw[light-gray](5,1) rectangle (6,2);
\filldraw[dark-gray](5,2) rectangle (6,3);
\filldraw[dark-gray](5,3) rectangle (6,4);
\filldraw[light-gray](5,4) rectangle (6,5);
\filldraw[dark-gray](5,5) rectangle (6,6);
\draw[black, fill=black] (1,1) circle (0.2);
\draw[black, fill=black] (2,5) circle (0.2);
\draw[black, fill=black] (3,3) circle (0.2);
\draw[black, fill=black] (4,4) circle (0.2);
\draw[black, fill=black] (5,2) circle (0.2);
\draw[thick](0,0)--(0,6);
\draw[thick](0,0)--(6,0);
\draw[thick](1,0)--(1,6);
\draw[thick](0,1)--(6,1);
\draw[thick](2,0)--(2,6);
\draw[thick](0,2)--(6,2);
\draw[thick](3,0)--(3,6);
\draw[thick](0,3)--(6,3);
\draw[thick](4,0)--(4,6);
\draw[thick](0,4)--(6,4);
\draw[thick](5,0)--(5,6);
\draw[thick](0,5)--(6,5);
\draw[thick](6,0)--(6,6);
\draw[thick](0,6)--(6,6);
\draw[ultra thick] (3.2,2.8)--(4.8,2.2);
\draw[ultra thick] (3.2, 3.2)--(3.8, 3.8);
\end{tikzpicture}
				\end{center}
				\caption{}
				\label{figure:wedge25}
		        \end{subfigure}
		        \begin{subfigure}[b]{0.24\textwidth}
				\begin{center}
\begin{tikzpicture}[scale=.4]
\filldraw[light-gray](0,0) rectangle (1,1);
\draw[thick](0,5)--(0,5);
\filldraw[dark-gray](0,3) rectangle (1,4);
\filldraw[dark-gray](0,4) rectangle (1,5);
\draw[thick](0,2)--(0,2);
\filldraw[dark-gray](0,6) rectangle (1,7);
\filldraw[light-gray](1,0) rectangle (2,1);
\filldraw[dark-gray](1,4) rectangle (2,5);
\draw[thick](1,2)--(1,2);
\filldraw[dark-gray](1,6) rectangle (2,7);
\filldraw[light-gray](2,0) rectangle (3,1);
\filldraw[light-gray](2,1) rectangle (3,2);
\filldraw[light-gray](2,2) rectangle (3,3);
\filldraw[light-gray](2,3) rectangle (3,4);
\filldraw[dark-gray](2,4) rectangle (3,5);
\filldraw[dark-gray](2,5) rectangle (3,6);
\filldraw[light-gray](2,6) rectangle (3,7);
\filldraw[dark-gray](3,0) rectangle (4,1);
\filldraw[dark-gray](3,1) rectangle (4,2);
\filldraw[light-gray](3,2) rectangle (4,3);
\filldraw[light-gray](3,3) rectangle (4,4);
\filldraw[light-gray](3,4) rectangle (4,5);
\filldraw[dark-gray](3,5) rectangle (4,6);
\filldraw[light-gray](3,6) rectangle (4,7);
\filldraw[dark-gray](4,0) rectangle (5,1);
\filldraw[dark-gray](4,1) rectangle (5,2);
\draw[thick](4,5)--(4,5);
\filldraw[light-gray](4,3) rectangle (5,4);
\filldraw[light-gray](4,5) rectangle (5,6);
\filldraw[dark-gray](4,6) rectangle (5,7);
\filldraw[dark-gray](5,0) rectangle (6,1);
\filldraw[dark-gray](5,1) rectangle (6,2);
\draw[thick](5,5)--(5,5);
\filldraw[light-gray](5,3) rectangle (6,4);
\filldraw[light-gray](5,4) rectangle (6,5);
\filldraw[light-gray](5,5) rectangle (6,6);
\filldraw[dark-gray](5,6) rectangle (6,7);
\filldraw[light-gray](6,0) rectangle (7,1);
\filldraw[light-gray](6,1) rectangle (7,2);
\filldraw[dark-gray](6,2) rectangle (7,3);
\filldraw[dark-gray](6,3) rectangle (7,4);
\filldraw[dark-gray](6,4) rectangle (7,5);
\filldraw[light-gray](6,5) rectangle (7,6);
\filldraw[dark-gray](6,6) rectangle (7,7);
\draw[black, fill=black] (1,1) circle (0.2);
\draw[black, fill=black] (2,6) circle (0.2);
\draw[black, fill=black] (3,3) circle (0.2);
\draw[black, fill=black] (4,4) circle (0.2);
\draw[black, fill=black] (5,5) circle (0.2);
\draw[black, fill=black] (6,2) circle (0.2);
\draw[thick](0,0)--(0,7);
\draw[thick](0,0)--(7,0);
\draw[thick](1,0)--(1,7);
\draw[thick](0,1)--(7,1);
\draw[thick](2,0)--(2,7);
\draw[thick](0,2)--(7,2);
\draw[thick](3,0)--(3,7);
\draw[thick](0,3)--(7,3);
\draw[thick](4,0)--(4,7);
\draw[thick](0,4)--(7,4);
\draw[thick](5,0)--(5,7);
\draw[thick](0,5)--(7,5);
\draw[thick](6,0)--(6,7);
\draw[thick](0,6)--(7,6);
\draw[thick](7,0)--(7,7);
\draw[thick](0,7)--(7,7);
\draw[ultra thick] (4.2,2.8)--(5.8,2.2);
\draw[ultra thick] (4.2, 4.2)--(4.8, 4.8);
\end{tikzpicture}
				\end{center}
				\caption{}
				\label{figure:wedge26}
		        \end{subfigure}
		        \begin{subfigure}[b]{0.24\textwidth}
				\begin{center}
		\begin{tikzpicture}[scale=.4]
\filldraw[light-gray](0,0) rectangle (1,1);
\draw[thick](0,6)--(0,6);
\filldraw[dark-gray](0,3) rectangle (1,4);
\filldraw[dark-gray](0,4) rectangle (1,5);
\filldraw[dark-gray](0,5) rectangle (1,6);
\filldraw[dark-gray](0,6) rectangle (1,7);
\filldraw[dark-gray](0,7) rectangle (1,8);
\filldraw[light-gray](1,0) rectangle (2,1);
\filldraw[light-gray](1,3) rectangle (2,4);
\filldraw[light-gray](1,4) rectangle (2,5);
\filldraw[dark-gray](1,5) rectangle (2,6);
\draw[thick](1,2)--(1,2);
\filldraw[dark-gray](1,7) rectangle (2,8);
\filldraw[dark-gray](2,0) rectangle (3,1);
\filldraw[dark-gray](2,1) rectangle (3,2);
\filldraw[dark-gray](2,2) rectangle (3,3);
\draw[thick](2,5)--(2,5);
\filldraw[dark-gray](2,5) rectangle (3,6);
\draw[thick](2,2)--(2,2);
\filldraw[dark-gray](2,7) rectangle (3,8);
\filldraw[dark-gray](3,0) rectangle (4,1);
\filldraw[dark-gray](3,1) rectangle (4,2);
\filldraw[dark-gray](3,2) rectangle (4,3);
\filldraw[light-gray](3,3) rectangle (4,4);
\filldraw[light-gray](3,4) rectangle (4,5);
\filldraw[dark-gray](3,5) rectangle (4,6);
\filldraw[dark-gray](3,6) rectangle (4,7);
\filldraw[light-gray](3,7) rectangle (4,8);
\filldraw[dark-gray](4,0) rectangle (5,1);
\filldraw[dark-gray](4,1) rectangle (5,2);
\filldraw[light-gray](4,2) rectangle (5,3);
\filldraw[dark-gray](4,3) rectangle (5,4);
\filldraw[light-gray](4,4) rectangle (5,5);
\filldraw[light-gray](4,5) rectangle (5,6);
\filldraw[dark-gray](4,6) rectangle (5,7);
\filldraw[light-gray](4,7) rectangle (5,8);
\filldraw[dark-gray](5,0) rectangle (6,1);
\filldraw[dark-gray](5,1) rectangle (6,2);
\draw[thick](5,6)--(5,6);
\filldraw[dark-gray](5,3) rectangle (6,4);
\filldraw[light-gray](5,4) rectangle (6,5);
\filldraw[light-gray](5,6) rectangle (6,7);
\filldraw[dark-gray](5,7) rectangle (6,8);
\filldraw[dark-gray](6,0) rectangle (7,1);
\filldraw[dark-gray](6,1) rectangle (7,2);
\draw[thick](6,6)--(6,6);
\filldraw[light-gray](6,3) rectangle (7,4);
\filldraw[light-gray](6,4) rectangle (7,5);
\filldraw[light-gray](6,5) rectangle (7,6);
\filldraw[light-gray](6,6) rectangle (7,7);
\filldraw[dark-gray](6,7) rectangle (7,8);
\filldraw[light-gray](7,0) rectangle (8,1);
\filldraw[light-gray](7,1) rectangle (8,2);
\filldraw[dark-gray](7,2) rectangle (8,3);
\filldraw[dark-gray](7,3) rectangle (8,4);
\filldraw[dark-gray](7,4) rectangle (8,5);
\filldraw[dark-gray](7,5) rectangle (8,6);
\filldraw[light-gray](7,6) rectangle (8,7);
\filldraw[dark-gray](7,7) rectangle (8,8);
\draw[black, fill=black] (1,1) circle (0.2);
\draw[black, fill=black] (2,4) circle (0.2);
\draw[black, fill=black] (3,7) circle (0.2);
\draw[black, fill=black] (4,3) circle (0.2);
\draw[black, fill=black] (5,5) circle (0.2);
\draw[black, fill=black] (6,6) circle (0.2);
\draw[black, fill=black] (7,2) circle (0.2);
\draw[thick](0,0)--(0,8);
\draw[thick](0,0)--(8,0);
\draw[thick](1,0)--(1,8);
\draw[thick](0,1)--(8,1);
\draw[thick](2,0)--(2,8);
\draw[thick](0,2)--(8,2);
\draw[thick](3,0)--(3,8);
\draw[thick](0,3)--(8,3);
\draw[thick](4,0)--(4,8);
\draw[thick](0,4)--(8,4);
\draw[thick](5,0)--(5,8);
\draw[thick](0,5)--(8,5);
\draw[thick](6,0)--(6,8);
\draw[thick](0,6)--(8,6);
\draw[thick](7,0)--(7,8);
\draw[thick](0,7)--(8,7);
\draw[thick](8,0)--(8,8);
\draw[thick](0,8)--(8,8);
\draw[ultra thick] (5.2,2.8)--(6.8,2.2);
\draw[ultra thick] (5.2, 5.2)--(5.8, 5.8);
\node at (1.5, 6.5) {$A$};
\node at (5.3, 2.35) {\footnotesize $B$};
\end{tikzpicture}
				\end{center}
				\caption{}
				\label{figure:wedge27}
		        \end{subfigure}%
		        \begin{subfigure}[b]{0.24\textwidth}
				\begin{center}
	\begin{tikzpicture}[scale=.4]
\filldraw[light-gray](0,0) rectangle (1,1);
\draw[thick](0,7)--(0,7);
\filldraw[dark-gray](0,3) rectangle (1,4);
\filldraw[dark-gray](0,4) rectangle (1,5);
\filldraw[dark-gray](0,5) rectangle (1,6);
\filldraw[dark-gray](0,6) rectangle (1,7);
\filldraw[dark-gray](0,7) rectangle (1,8);
\filldraw[dark-gray](0,8) rectangle (1,9);
\filldraw[light-gray](1,0) rectangle (2,1);
\filldraw[light-gray](1,4) rectangle (2,5);
\filldraw[light-gray](1,5) rectangle (2,6);
\filldraw[dark-gray](1,6) rectangle (2,7);
\draw[thick](1,2)--(1,2);
\filldraw[dark-gray](1,8) rectangle (2,9);
\filldraw[dark-gray](2,0) rectangle (3,1);
\filldraw[dark-gray](2,1) rectangle (3,2);
\filldraw[dark-gray](2,2) rectangle (3,3);
\filldraw[dark-gray](2,3) rectangle (3,4);
\draw[thick](2,5)--(2,5);
\filldraw[dark-gray](2,6) rectangle (3,7);
\draw[thick](2,2)--(2,2);
\filldraw[dark-gray](2,8) rectangle (3,9);
\filldraw[dark-gray](3,0) rectangle (4,1);
\filldraw[dark-gray](3,1) rectangle (4,2);
\filldraw[dark-gray](3,2) rectangle (4,3);
\filldraw[dark-gray](3,3) rectangle (4,4);
\filldraw[light-gray](3,4) rectangle (4,5);
\filldraw[light-gray](3,5) rectangle (4,6);
\filldraw[dark-gray](3,6) rectangle (4,7);
\filldraw[dark-gray](3,7) rectangle (4,8);
\filldraw[light-gray](3,8) rectangle (4,9);
\filldraw[dark-gray](4,0) rectangle (5,1);
\filldraw[dark-gray](4,1) rectangle (5,2);
\filldraw[dark-gray](4,2) rectangle (5,3);
\filldraw[light-gray](4,3) rectangle (5,4);
\filldraw[dark-gray](4,4) rectangle (5,5);
\filldraw[light-gray](4,5) rectangle (5,6);
\filldraw[light-gray](4,6) rectangle (5,7);
\filldraw[dark-gray](4,7) rectangle (5,8);
\filldraw[light-gray](4,8) rectangle (5,9);
\filldraw[dark-gray](5,0) rectangle (6,1);
\filldraw[dark-gray](5,1) rectangle (6,2);
\filldraw[dark-gray](5,2) rectangle (6,3);
\draw[thick](5,6)--(5,6);
\filldraw[dark-gray](5,4) rectangle (6,5);
\filldraw[light-gray](5,5) rectangle (6,6);
\filldraw[dark-gray](5,7) rectangle (6,8);
\filldraw[dark-gray](5,8) rectangle (6,9);
\filldraw[dark-gray](6,0) rectangle (7,1);
\filldraw[dark-gray](6,1) rectangle (7,2);
\filldraw[light-gray](6,2) rectangle (7,3);
\filldraw[dark-gray](6,3) rectangle (7,4);
\filldraw[dark-gray](6,4) rectangle (7,5);
\filldraw[dark-gray](6,5) rectangle (7,6);
\filldraw[light-gray](6,7) rectangle (7,8);
\filldraw[dark-gray](6,8) rectangle (7,9);
\filldraw[dark-gray](7,0) rectangle (8,1);
\filldraw[dark-gray](7,1) rectangle (8,2);
\draw[thick](7,7)--(7,7);
\filldraw[dark-gray](7,3) rectangle (8,4);
\filldraw[dark-gray](7,4) rectangle (8,5);
\filldraw[dark-gray](7,5) rectangle (8,6);
\filldraw[light-gray](7,6) rectangle (8,7);
\filldraw[light-gray](7,7) rectangle (8,8);
\filldraw[dark-gray](7,8) rectangle (8,9);
\filldraw[light-gray](8,0) rectangle (9,1);
\filldraw[light-gray](8,1) rectangle (9,2);
\filldraw[dark-gray](8,2) rectangle (9,3);
\filldraw[dark-gray](8,3) rectangle (9,4);
\filldraw[dark-gray](8,4) rectangle (9,5);
\filldraw[dark-gray](8,5) rectangle (9,6);
\filldraw[dark-gray](8,6) rectangle (9,7);
\filldraw[light-gray](8,7) rectangle (9,8);
\filldraw[dark-gray](8,8) rectangle (9,9);
\draw[black, fill=black] (1,1) circle (0.2);
\draw[black, fill=black] (2,5) circle (0.2);
\draw[black, fill=black] (3,8) circle (0.2);
\draw[black, fill=black] (4,4) circle (0.2);
\draw[black, fill=black] (5,6) circle (0.2);
\draw[black, fill=black] (6,3) circle (0.2);
\draw[black, fill=black] (7,7) circle (0.2);
\draw[black, fill=black] (8,2) circle (0.2);
\draw[thick](0,0)--(0,9);
\draw[thick](0,0)--(9,0);
\draw[thick](1,0)--(1,9);
\draw[thick](0,1)--(9,1);
\draw[thick](2,0)--(2,9);
\draw[thick](0,2)--(9,2);
\draw[thick](3,0)--(3,9);
\draw[thick](0,3)--(9,3);
\draw[thick](4,0)--(4,9);
\draw[thick](0,4)--(9,4);
\draw[thick](5,0)--(5,9);
\draw[thick](0,5)--(9,5);
\draw[thick](6,0)--(6,9);
\draw[thick](0,6)--(9,6);
\draw[thick](7,0)--(7,9);
\draw[thick](0,7)--(9,7);
\draw[thick](8,0)--(8,9);
\draw[thick](0,8)--(9,8);
\draw[thick](9,0)--(9,9);
\draw[thick](0,9)--(9,9);
\draw[ultra thick] (5.2,3.8)--(5.8,3.2);
\draw[ultra thick] (7.2,2.8)--(7.8,2.2);
\draw[ultra thick] (5.2, 6.2)--(6.8, 6.8);
\node at (1.5, 7.5) {$A$};
\node at (1.5, 3.5) {$C$};
\end{tikzpicture}
				\end{center}
				\caption{}
				\label{figure:wedge28}
		        \end{subfigure}
%

		\vspace*{.25cm}		        
			 \begin{subfigure}[b]{0.24\textwidth}
				\begin{center}
					\begin{tikzpicture}[scale=.4]
						\filldraw[dark-gray](0,3) rectangle (1,4);
						\filldraw[dark-gray](1,2) rectangle (2,3);
						\filldraw[light-gray](1,3) rectangle (2,4);
						\filldraw[light-gray](2,2) rectangle (3,3);
						\filldraw[light-gray](2,3) rectangle (3,4);
						\filldraw[light-gray](3,0) rectangle (4,1);
						\filldraw[light-gray](3,1) rectangle (4,2);
						\filldraw[light-gray](3,2) rectangle (4,3);
						\filldraw[light-gray](1,0) rectangle (2,1);
						\filldraw[light-gray](2,0) rectangle (3,1);
						\filldraw[dark-gray](3,3) rectangle (4,4);
						\draw[black, fill=black] (1,3) circle (0.2);
						\draw[black, fill=black] (2,1) circle (0.2);
						\draw[black, fill=black] (3,2) circle (0.2);
						\draw[thick](0,0)--(0,4);
						\draw[thick](0,0)--(4,0);
						\draw[thick](1,0)--(1,4);
						\draw[thick](0,1)--(4,1);
						\draw[thick](2,0)--(2,4);
						\draw[thick](0,2)--(4,2);
						\draw[thick](3,0)--(3,4);
						\draw[thick](0,3)--(4,3);
						\draw[thick](4,0)--(4,4);
						\draw[thick](0,4)--(4,4);
						\node at (.5,2.5) {$E$};
						\node at (.5,1.5) {$F$};
						\node at (1.5,1.5) {$G$};
					\end{tikzpicture}
				\end{center}
				\caption{}
				\label{figure:simples-case1-7}
		        \end{subfigure}
		        \begin{subfigure}[b]{0.24\textwidth}
				\begin{center}
					\begin{tikzpicture}[scale=.4]
						\filldraw[light-gray](0,1) rectangle (1,2);
						\filldraw[light-gray](0,2) rectangle (1,3);
						\filldraw[dark-gray](0,4) rectangle (1,5);
						\filldraw[dark-gray](1,0) rectangle (2,1);
						\filldraw[dark-gray](1,4) rectangle (2,5);
						\filldraw[dark-gray](2,0) rectangle (3,1);
						\filldraw[dark-gray](2,3) rectangle (3,4);
						\filldraw[light-gray](2,4) rectangle (3,5);
						\filldraw[light-gray](3,0) rectangle (4,1);
						\filldraw[dark-gray](3,1) rectangle (4,2);
						\filldraw[light-gray](3,3) rectangle (4,4);
						\filldraw[light-gray](3,4) rectangle (4,5);
						\filldraw[light-gray](4,0) rectangle (5,1);
						\filldraw[light-gray](4,1) rectangle (5,2);
						\filldraw[light-gray](4,2) rectangle (5,3);
						\filldraw[light-gray](4,3) rectangle (5,4);
						\filldraw[dark-gray](4,4) rectangle (5,5);
						\draw[black, fill=black] (1,2) circle (0.2);
						\draw[black, fill=black] (2,4) circle (0.2);
						\draw[black, fill=black] (3,1) circle (0.2);
						\draw[black, fill=black] (4,3) circle (0.2);
						\draw[thick](0,0)--(0,5);
						\draw[thick](0,0)--(5,0);
						\draw[thick](1,0)--(1,5);
						\draw[thick](0,1)--(5,1);
						\draw[thick](2,0)--(2,5);
						\draw[thick](0,2)--(5,2);
						\draw[thick](3,0)--(3,5);
						\draw[thick](0,3)--(5,3);
						\draw[thick](4,0)--(4,5);
						\draw[thick](0,4)--(5,4);
						\draw[thick](5,0)--(5,5);
						\draw[thick](0,5)--(5,5);
						\draw[ultra thick] (0.2,3.2)--(1.8,3.8);
						\node at (.5,.5) {$H$};
						\node at (1.5,1.5) {$I$};
						\node at (2.5,1.5) {$J$};
						\node at (3.5,2.5) {$K$};
					\end{tikzpicture}
				\end{center}
				\caption{}
				\label{figure:simples-case1-8}
		        \end{subfigure}
		        \begin{subfigure}[b]{0.24\textwidth}
				\begin{center}
					\begin{tikzpicture}[scale=.4]
						\filldraw[light-gray](0,1) rectangle (1,2);
						\filldraw[light-gray](0,2) rectangle (1,3);
						\filldraw[dark-gray](0,4) rectangle (1,5);
						\filldraw[dark-gray](1,0) rectangle (2,1);
						\filldraw[dark-gray](1,4) rectangle (2,5);
						\filldraw[dark-gray](2,0) rectangle (3,1);
						\filldraw[dark-gray](2,3) rectangle (3,4);
						\filldraw[light-gray](2,4) rectangle (3,5);
						\filldraw[light-gray](3,0) rectangle (4,1);
						\filldraw[dark-gray](3,1) rectangle (4,2);
						\filldraw[light-gray](3,3) rectangle (4,4);
						\filldraw[light-gray](3,4) rectangle (4,5);
						\filldraw[light-gray](4,0) rectangle (5,1);
						\filldraw[light-gray](4,1) rectangle (5,2);
						\filldraw[light-gray](4,2) rectangle (5,3);
						\filldraw[light-gray](4,3) rectangle (5,4);
						\filldraw[light-gray](0,0) rectangle (1,1);
						\filldraw[dark-gray](4,4) rectangle (5,5);
						\draw[black, fill=black] (1,2) circle (0.2);
						\draw[black, fill=black] (2,4) circle (0.2);
						\draw[black, fill=black] (3,1) circle (0.2);
						\draw[black, fill=black] (4,3) circle (0.2);
						\draw[thick](0,0)--(0,5);
						\draw[thick](0,0)--(5,0);
						\draw[thick](1,0)--(1,5);
						\draw[thick](0,1)--(5,1);
						\draw[thick](2,0)--(2,5);
						\draw[thick](0,2)--(5,2);
						\draw[thick](3,0)--(3,5);
						\draw[thick](0,3)--(5,3);
						\draw[thick](4,0)--(4,5);
						\draw[thick](0,4)--(5,4);
						\draw[thick](5,0)--(5,5);
						\draw[thick](0,5)--(5,5);
						\node at (.3, 3.7) {$_L$};
						\draw[ultra thick] (0.2,3.2)--(1.8,3.8);
						\draw[ultra thick] (1.2,1.8)--(2.8,1.2);
						\draw[ultra thick] (3.2,2.2)--(3.8,2.8);
					\end{tikzpicture}
				\end{center}	
				\caption{}
				\label{figure:simples-case1-9}
		        \end{subfigure}
		          \begin{subfigure}[b]{0.24\textwidth}
				\begin{center}
				\begin{tikzpicture}[scale=.4]
					\filldraw[light-gray](0,0) rectangle (1,1);
					\filldraw[light-gray](0,1) rectangle (1,2);
					\filldraw[light-gray](0,2) rectangle (1,3);
					\filldraw[light-gray](0,3) rectangle (1,4);
					\filldraw[dark-gray](0,4) rectangle (1,5);
					\filldraw[dark-gray](0,5) rectangle (1,6);
					\filldraw[dark-gray](1,0) rectangle (2,1);
					\filldraw[dark-gray](1,1) rectangle (2,2);
					\filldraw[light-gray](1,2) rectangle (2,3);
					\filldraw[dark-gray](1,3) rectangle (2,4);
					\filldraw[dark-gray](1,5) rectangle (2,6);
					\filldraw[dark-gray](2,0) rectangle (3,1);
					\filldraw[dark-gray](2,2) rectangle (3,3);
					\filldraw[dark-gray](2,3) rectangle (3,4);
					\filldraw[dark-gray](2,5) rectangle (3,6);
					\filldraw[dark-gray](3,0) rectangle (4,1);
					\filldraw[dark-gray](3,3) rectangle (4,4);
					\filldraw[dark-gray](3,4) rectangle (4,5);
					\filldraw[light-gray](3,5) rectangle (4,6);
					\filldraw[light-gray](4,0) rectangle (5,1);
					\filldraw[dark-gray](4,1) rectangle (5,2);
					\filldraw[light-gray](4,3) rectangle (5,4);
					\filldraw[light-gray](4,4) rectangle (5,5);
					\filldraw[light-gray](4,5) rectangle (5,6);
					\filldraw[light-gray](5,0) rectangle (6,1);
					\filldraw[dark-gray](5,1) rectangle (6,2);
					\filldraw[light-gray](5,2) rectangle (6,3);
					\filldraw[light-gray](5,3) rectangle (6,4);
					\filldraw[dark-gray](5,4) rectangle (6,5);
					\filldraw[dark-gray](5,5) rectangle (6,6);
					\draw[black, fill=black] (1,4) circle (0.2);
					\draw[black, fill=black] (2,2) circle (0.2);
					\draw[black, fill=black] (3,5) circle (0.2);
					\draw[black, fill=black] (4,1) circle (0.2);
					\draw[black, fill=black] (5,3) circle (0.2);
					\draw[thick](0,0)--(0,6);
					\draw[thick](0,0)--(6,0);
					\draw[thick](1,0)--(1,6);
					\draw[thick](0,1)--(6,1);
					\draw[thick](2,0)--(2,6);
					\draw[thick](0,2)--(6,2);
					\draw[thick](3,0)--(3,6);
					\draw[thick](0,3)--(6,3);
					\draw[thick](4,0)--(4,6);
					\draw[thick](0,4)--(6,4);
					\draw[thick](5,0)--(5,6);
					\draw[thick](0,5)--(6,5);
					\draw[thick](6,0)--(6,6);
					\draw[thick](0,6)--(6,6);
					\draw[ultra thick] (2.2,1.8)--(3.8,1.2);
					\draw[ultra thick] (1.2,4.2)--(2.8,4.8);
					\draw[ultra thick] (3.2,2.2)--(4.8,2.8);
				\end{tikzpicture}
			\end{center}
				\caption{}
				\label{figure:simples-case1-10}
		        \end{subfigure}
		         
		        \caption{Permutation diagrams corresponding to steps in the proof of Theorem \ref{theorem:simples}.}
		\end{figure}
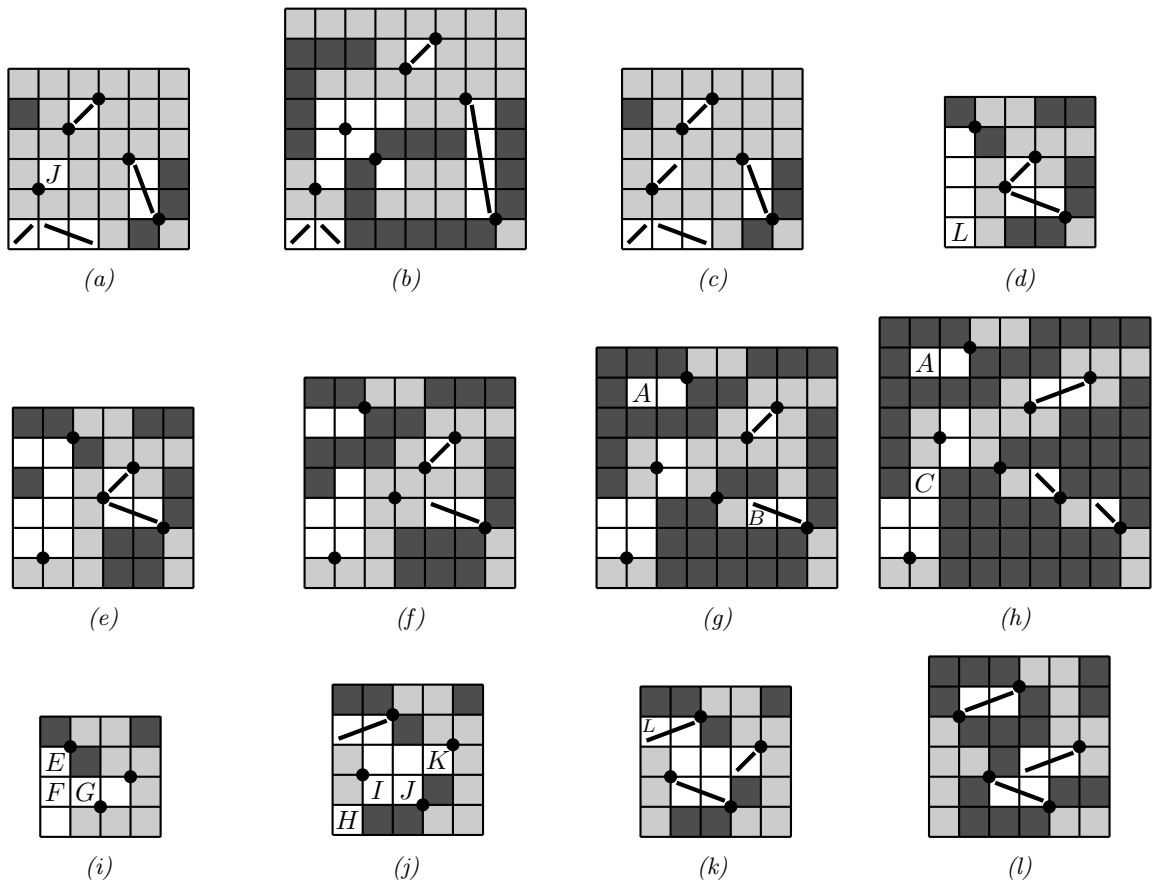

		Hence, we can now assume that $\pi(n)$ is in the upper part of the wedge. Assume that $\pi_R$ has at least two entries, and let the second entry be as low as possible. Thus, we have the diagram in Figure~\ref{figure:simples-case1-7}. If $E$ has a descent, then there is a $4312$ pattern. If $F$ is empty, then $G$ must have an entry, and it follows by the usual arguments that the rectangular hull of $G$ together with the bottommost entry of $\pi_R$ forms an interval. Hence $E$ is increasing and there must be an entry in $F$. Select the leftmost possible entry of $F$. See Figure~\ref{figure:simples-case1-8}.
		
	Any entry in $H$ has no entry above it and to its left, since there would then be a $3124$ pattern. This implies that cell $H$ is actually empty because otherwise the rectangular hull of cell $H$ is itself an interval of $\pi$. The cells $I$ and $J$ together must be decreasing to avoid a $3124$ pattern. Cell $K$ must be increasing, because of the known wedge shape of $\pi_R$. Hence, we have the diagram in Figure~\ref{figure:simples-case1-9}. Suppose that cell $L$ has some entry. Place it as far to the left as possible. This gives us the diagram in Figure~\ref{figure:simples-case1-10}. Any permutation drawn on this figure must lie in $\GG_2$.

Now suppose instead that cell $L$ in Figure~\ref{figure:simples-case1-9} is actually empty. This yields the diagram in Figure~\ref{figure:simples-case1-11}. If cell $M$ is empty, we get the diagram in Figure~\ref{figure:simples-case1-12}. Any permutation drawn on Figure~\ref{figure:simples-case1-12} must lie in $\GG_2$. So, suppose instead that cell $M$ has an entry, and place it as high as possible. This gives us Figure~\ref{figure:simples-case1-13}. Cell $N$ must be empty to avoid an unsplittable interval contained in the rectangular hull of $\pi(1)$ and cell $N$. Additionally, any entry of $O$ must be to the left of any entry in $P$, otherwise a $3124$ pattern occurs. By the usual monotonicity arguments, we get the diagram in Figure~\ref{figure:simples-case1-14}. Any permutation drawn on this figure lies in the class $\GG_2$. This completes the first case.

		\begin{figure}
		        \centering
		        \begin{subfigure}[b]{0.24\textwidth}
				\begin{center}
					\begin{tikzpicture}[scale=.4]
						\filldraw[light-gray](0,0) rectangle (1,1);
						\filldraw[light-gray](0,1) rectangle (1,2);
						\filldraw[light-gray](0,2) rectangle (1,3);
						\filldraw[light-gray](0,3) rectangle (1,4);
						\filldraw[dark-gray](0,4) rectangle (1,5);
						\filldraw[dark-gray](1,0) rectangle (2,1);
						\filldraw[dark-gray](1,4) rectangle (2,5);
						\filldraw[dark-gray](2,0) rectangle (3,1);
						\filldraw[dark-gray](2,3) rectangle (3,4);
						\filldraw[light-gray](2,4) rectangle (3,5);
						\filldraw[light-gray](3,0) rectangle (4,1);
						\filldraw[dark-gray](3,1) rectangle (4,2);
						\filldraw[light-gray](3,3) rectangle (4,4);
						\filldraw[light-gray](3,4) rectangle (4,5);
						\filldraw[light-gray](4,0) rectangle (5,1);
						\filldraw[light-gray](4,1) rectangle (5,2);
						\filldraw[light-gray](4,2) rectangle (5,3);
						\filldraw[light-gray](4,3) rectangle (5,4);
						\filldraw[dark-gray](4,4) rectangle (5,5);
						\draw[black, fill=black] (1,2) circle (0.2);
						\draw[black, fill=black] (2,4) circle (0.2);
						\draw[black, fill=black] (3,1) circle (0.2);
						\draw[black, fill=black] (4,3) circle (0.2);
						\draw[thick](0,0)--(0,5);
						\draw[thick](0,0)--(5,0);
						\draw[thick](1,0)--(1,5);
						\draw[thick](0,1)--(5,1);
						\draw[thick](2,0)--(2,5);
						\draw[thick](0,2)--(5,2);
						\draw[thick](3,0)--(3,5);
						\draw[thick](0,3)--(5,3);
						\draw[thick](4,0)--(4,5);
						\draw[thick](0,4)--(5,4);
						\draw[thick](5,0)--(5,5);
						\draw[thick](0,5)--(5,5);
						\draw[ultra thick] (1.2,3.2)--(1.8,3.8);
						\draw[ultra thick] (1.2,1.8)--(2.8,1.2);
						\draw[ultra thick] (3.2,2.2)--(3.8,2.8);
						\node at (1.5,2.5) {$M$};
					\end{tikzpicture}
				\end{center}
				\caption{}
				\label{figure:simples-case1-11}

		        \end{subfigure}
		        \begin{subfigure}[b]{0.24\textwidth}
				\begin{center}
					\begin{tikzpicture}[scale=.4]
						\filldraw[light-gray](0,0) rectangle (1,1);
						\filldraw[light-gray](0,1) rectangle (1,2);
						\filldraw[light-gray](0,2) rectangle (1,3);
						\filldraw[light-gray](0,3) rectangle (1,4);
						\filldraw[dark-gray](0,4) rectangle (1,5);
						\filldraw[dark-gray](1,0) rectangle (2,1);
						\filldraw[dark-gray](1,4) rectangle (2,5);
						\filldraw[dark-gray](2,0) rectangle (3,1);
						\filldraw[dark-gray](2,3) rectangle (3,4);
						\filldraw[light-gray](2,4) rectangle (3,5);
						\filldraw[light-gray](3,0) rectangle (4,1);
						\filldraw[dark-gray](3,1) rectangle (4,2);
						\filldraw[light-gray](3,3) rectangle (4,4);
						\filldraw[light-gray](3,4) rectangle (4,5);
						\filldraw[light-gray](4,0) rectangle (5,1);
						\filldraw[light-gray](4,1) rectangle (5,2);
						\filldraw[light-gray](4,2) rectangle (5,3);
						\filldraw[light-gray](4,3) rectangle (5,4);
						\filldraw[light-gray](1,2) rectangle (2,3);
						\filldraw[dark-gray](4,4) rectangle (5,5);
						\draw[black, fill=black] (1,2) circle (0.2);
						\draw[black, fill=black] (2,4) circle (0.2);
						\draw[black, fill=black] (3,1) circle (0.2);
						\draw[black, fill=black] (4,3) circle (0.2);
						\draw[thick](0,0)--(0,5);
						\draw[thick](0,0)--(5,0);
						\draw[thick](1,0)--(1,5);
						\draw[thick](0,1)--(5,1);
						\draw[thick](2,0)--(2,5);
						\draw[thick](0,2)--(5,2);
						\draw[thick](3,0)--(3,5);
						\draw[thick](0,3)--(5,3);
						\draw[thick](4,0)--(4,5);
						\draw[thick](0,4)--(5,4);
						\draw[thick](5,0)--(5,5);
						\draw[thick](0,5)--(5,5);
						\draw[ultra thick] (1.2,3.2)--(1.8,3.8);
						\draw[ultra thick] (1.2,1.8)--(2.8,1.2);
						\draw[ultra thick] (2.2,2.2)--(3.8,2.8);
					\end{tikzpicture}
				\end{center}
				\caption{}
				\label{figure:simples-case1-12}
		        \end{subfigure}
		        \begin{subfigure}[b]{0.24\textwidth}
				\begin{center}
					\begin{tikzpicture}[scale=.4]
						\filldraw[light-gray](0,0) rectangle (1,1);
						\filldraw[light-gray](0,1) rectangle (1,2);
						\filldraw[light-gray](0,2) rectangle (1,3);
						\filldraw[dark-gray](0,3) rectangle (1,4);
						\filldraw[dark-gray](0,4) rectangle (1,5);
						\filldraw[dark-gray](0,5) rectangle (1,6);
						\filldraw[dark-gray](1,0) rectangle (2,1);
						\filldraw[light-gray](1,3) rectangle (2,4);
						\filldraw[dark-gray](1,5) rectangle (2,6);
						\filldraw[dark-gray](2,0) rectangle (3,1);
						\filldraw[light-gray](2,3) rectangle (3,4);
						\filldraw[dark-gray](2,5) rectangle (3,6);
						\filldraw[dark-gray](3,0) rectangle (4,1);
						\filldraw[dark-gray](3,4) rectangle (4,5);
						\filldraw[light-gray](3,5) rectangle (4,6);
						\filldraw[light-gray](4,0) rectangle (5,1);
						\filldraw[dark-gray](4,1) rectangle (5,2);
						\filldraw[dark-gray](4,2) rectangle (5,3);
						\filldraw[light-gray](4,4) rectangle (5,5);
						\filldraw[light-gray](4,5) rectangle (5,6);
						\filldraw[light-gray](5,0) rectangle (6,1);
						\filldraw[light-gray](5,1) rectangle (6,2);
						\filldraw[light-gray](5,2) rectangle (6,3);
						\filldraw[light-gray](5,3) rectangle (6,4);
						\filldraw[light-gray](5,4) rectangle (6,5);
						\filldraw[dark-gray](5,5) rectangle (6,6);
						\draw[black, fill=black] (1,2) circle (0.2);
						\draw[black, fill=black] (2,3) circle (0.2);
						\draw[black, fill=black] (3,5) circle (0.2);
						\draw[black, fill=black] (4,1) circle (0.2);
						\draw[black, fill=black] (5,4) circle (0.2);
						\draw[thick](0,0)--(0,6);
						\draw[thick](0,0)--(6,0);
						\draw[thick](1,0)--(1,6);
						\draw[thick](0,1)--(6,1);
						\draw[thick](2,0)--(2,6);
						\draw[thick](0,2)--(6,2);
						\draw[thick](3,0)--(3,6);
						\draw[thick](0,3)--(6,3);
						\draw[thick](4,0)--(4,6);
						\draw[thick](0,4)--(6,4);
						\draw[thick](5,0)--(5,6);
						\draw[thick](0,5)--(6,5);
						\draw[thick](6,0)--(6,6);
						\draw[thick](0,6)--(6,6);
						\node at (1.5,1.5) {$N$};
						\node at (1.5,2.5) {$O$};
						\node at (1.5,4.5) {$P$};
					\end{tikzpicture}
				\end{center}
				\caption{}
				\label{figure:simples-case1-13}
		        \end{subfigure}
			\begin{subfigure}[b]{0.24\textwidth}
				\begin{center}
					\begin{tikzpicture}[scale=.4]
						\filldraw[light-gray](0,0) rectangle (1,1);
						\filldraw[light-gray](0,1) rectangle (1,2);
						\filldraw[light-gray](0,2) rectangle (1,3);
						\filldraw[dark-gray](0,3) rectangle (1,4);
						\filldraw[dark-gray](0,4) rectangle (1,5);
						\filldraw[dark-gray](0,5) rectangle (1,6);
						\filldraw[dark-gray](1,0) rectangle (2,1);
						\filldraw[light-gray](1,3) rectangle (2,4);
						\filldraw[dark-gray](1,5) rectangle (2,6);
						\filldraw[dark-gray](2,0) rectangle (3,1);
						\filldraw[light-gray](2,3) rectangle (3,4);
						\filldraw[dark-gray](2,5) rectangle (3,6);
						\filldraw[dark-gray](3,0) rectangle (4,1);
						\filldraw[dark-gray](3,4) rectangle (4,5);
						\filldraw[light-gray](3,5) rectangle (4,6);
						\filldraw[light-gray](4,0) rectangle (5,1);
						\filldraw[dark-gray](4,1) rectangle (5,2);
						\filldraw[dark-gray](4,2) rectangle (5,3);
						\filldraw[light-gray](4,4) rectangle (5,5);
						\filldraw[light-gray](4,5) rectangle (5,6);
						\filldraw[light-gray](5,0) rectangle (6,1);
						\filldraw[light-gray](5,1) rectangle (6,2);
						\filldraw[light-gray](5,2) rectangle (6,3);
						\filldraw[light-gray](5,3) rectangle (6,4);
						\filldraw[light-gray](5,4) rectangle (6,5);
						\filldraw[dark-gray](5,5) rectangle (6,6);
						\filldraw[light-gray](1,1) rectangle (2,2);
						\filldraw[dark-gray](1,4) rectangle (1.5,5);
						\draw[black, fill=black] (1,2) circle (0.2);
						\draw[black, fill=black] (2,3) circle (0.2);
						\draw[black, fill=black] (3,5) circle (0.2);
						\draw[black, fill=black] (4,1) circle (0.2);
						\draw[black, fill=black] (5,4) circle (0.2);
						\draw[thick](0,0)--(0,6);
						\draw[thick](0,0)--(6,0);
						\draw[thick](1,0)--(1,6);
						\draw[thick](0,1)--(6,1);
						\draw[thick](2,0)--(2,6);
						\draw[thick](0,2)--(6,2);
						\draw[thick](3,0)--(3,6);
						\draw[thick](0,3)--(6,3);
						\draw[thick](4,0)--(4,6);
						\draw[thick](0,4)--(6,4);
						\draw[thick](5,0)--(5,6);
						\draw[thick](0,5)--(6,5);
						\draw[thick](6,0)--(6,6);
						\draw[thick](0,6)--(6,6);
						\draw[thick, dashed] (1.5,0)--(1.5,6);
						\draw[ultra thick] (1.1,2.1)--(1.4,2.9);
						\draw[ultra thick] (1.7,4.2)--(2.8,4.8);
						\draw[ultra thick] (3.2,3.2)--(4.8,3.8);
						\draw[ultra thick] (2.2,2.8)--(3.8,1.2);
					\end{tikzpicture}
				\end{center}
				\caption{}
				\label{figure:simples-case1-14}
		        \end{subfigure}
	
				\vspace*{.25cm}		        

		        \begin{subfigure}[b]{0.24\textwidth}
				\begin{center}
					\begin{tikzpicture}[scale=.4]
						\filldraw[dark-gray](0,4) rectangle (1,5);
						\filldraw[light-gray](1,0) rectangle (2,1);
						\filldraw[dark-gray](1,2) rectangle (2,3);
						\filldraw[dark-gray](1,3) rectangle (2,4);
						\filldraw[light-gray](1,4) rectangle (2,5);
						\filldraw[light-gray](2,0) rectangle (3,1);
						\filldraw[light-gray](2,3) rectangle (3,4);
						\filldraw[light-gray](2,4) rectangle (3,5);
						\filldraw[dark-gray](3,0) rectangle (4,1);
						\filldraw[dark-gray](3,1) rectangle (4,2);
						\filldraw[light-gray](3,3) rectangle (4,4);
						\filldraw[dark-gray](3,4) rectangle (4,5);
						\filldraw[light-gray](4,1) rectangle (5,2);
						\filldraw[dark-gray](4,2) rectangle (5,3);
						\filldraw[dark-gray](4,4) rectangle (5,5);
						\draw[black, fill=black] (1,4) circle (0.2);
						\draw[black, fill=black] (2,1) circle (0.2);
						\draw[black, fill=black] (3,3) circle (0.2);
						\draw[black, fill=black] (4,2) circle (0.2);
						\draw[thick](0,0)--(0,5);
						\draw[thick](0,0)--(5,0);
						\draw[thick](1,0)--(1,5);
						\draw[thick](0,1)--(5,1);
						\draw[thick](2,0)--(2,5);
						\draw[thick](0,2)--(5,2);
						\draw[thick](3,0)--(3,5);
						\draw[thick](0,3)--(5,3);
						\draw[thick](4,0)--(4,5);
						\draw[thick](0,4)--(5,4);
						\draw[thick](5,0)--(5,5);
						\draw[thick](0,5)--(5,5);
						\node at (.5,2.5) {$A$};
						\node at (2.5,2.5) {$B$};
					\end{tikzpicture}
				\end{center}
				\caption{}
				\label{figure:simples-case2-1}
		        \end{subfigure}
		        \begin{subfigure}[b]{0.24\textwidth}
				\begin{center}
					\begin{tikzpicture}[scale=.4]
						\clip (-.03,-.03) rectangle (6.03,6.03);
						\filldraw[dark-gray](0,3) rectangle (1,4);
						\filldraw[dark-gray](0,4) rectangle (1,5);
						\filldraw[dark-gray](0,5) rectangle (1,6);
						\filldraw[dark-gray](1,0) rectangle (2,1);
						\filldraw[dark-gray](1,2) rectangle (2,3);
						\filldraw[dark-gray](1,5) rectangle (2,6);
						\filldraw[dark-gray](2,0) rectangle (3,1);
						\filldraw[dark-gray](2,2) rectangle (3,3);
						\filldraw[dark-gray](2,3) rectangle (3,4);
						\filldraw[dark-gray](2,4) rectangle (3,5);
						\filldraw[light-gray](2,5) rectangle (3,6);
						\filldraw[light-gray](3,0) rectangle (4,1);
						\filldraw[dark-gray](3,1) rectangle (4,2);
						\filldraw[dark-gray](3,2) rectangle (4,3);
						\filldraw[light-gray](3,3) rectangle (4,4);
						\filldraw[light-gray](3,4) rectangle (4,5);
						\filldraw[light-gray](3,5) rectangle (4,6);
						\filldraw[dark-gray](4,0) rectangle (5,1);
						\filldraw[dark-gray](4,1) rectangle (5,2);
						\filldraw[light-gray](4,4) rectangle (5,5);
						\filldraw[dark-gray](4,5) rectangle (5,6);
						\filldraw[light-gray](5,1) rectangle (6,2);
						\filldraw[dark-gray](5,2) rectangle (6,3);
						\filldraw[dark-gray](5,3) rectangle (6,4);
						\filldraw[dark-gray](5,4) rectangle (6,5);
						\filldraw[dark-gray](5,5) rectangle (6,6);
						\filldraw[light-gray](0,2) rectangle (1,3);
						\draw[black, fill=black] (1,3) circle (0.2);
						\draw[black, fill=black] (2,5) circle (0.2);
						\draw[black, fill=black] (3,1) circle (0.2);
						\draw[black, fill=black] (4,4) circle (0.2);
						\draw[black, fill=black] (5,2) circle (0.2);
						\draw[thick](0,0)--(0,6);
						\draw[thick](0,0)--(6,0);
						\draw[thick](1,0)--(1,6);
						\draw[thick](0,1)--(6,1);
						\draw[thick](2,0)--(2,6);
						\draw[thick](0,2)--(6,2);
						\draw[thick](3,0)--(3,6);
						\draw[thick](0,3)--(6,3);
						\draw[thick](4,0)--(4,6);
						\draw[thick](0,4)--(6,4);
						\draw[thick](5,0)--(5,6);
						\draw[thick](0,5)--(6,5);
						\draw[thick](6,0)--(6,6);
						\draw[thick](0,6)--(6,6);
						\draw[ultra thick] (1.2,3.2)--(1.8,4.8);
						\draw[ultra thick] (4.2,3.8)--(4.8,2.2);
						\draw[ultra thick] (1.2,1.8)--(2.8,1.2);
						\draw[ultra thick] (5.2,0.8)--(5.8,0.2);
						\node at (5.27,0.15) {\footnotesize $^E$};
						\node at (0.5, 1.5) {$C$};
						\node at (0.5, 0.5) {$D$};
					\end{tikzpicture}
				\end{center}
				\caption{}
				\label{figure:simples-case2-2}
		        \end{subfigure}
		        \begin{subfigure}[b]{0.24\textwidth}
				\begin{center}
					\begin{tikzpicture}[scale=.4]
						\filldraw[light-gray](0,3) rectangle (1,4);
						\filldraw[dark-gray](0,4) rectangle (1,5);
						\filldraw[dark-gray](0,5) rectangle (1,6);
						\filldraw[dark-gray](0,6) rectangle (1,7);
						\filldraw[dark-gray](1,0) rectangle (2,1);
						\filldraw[dark-gray](1,1) rectangle (2,2);
						\filldraw[dark-gray](1,3) rectangle (2,4);
						\filldraw[dark-gray](1,6) rectangle (2,7);
						\filldraw[dark-gray](2,0) rectangle (3,1);
						\filldraw[dark-gray](2,1) rectangle (3,2);
						\filldraw[dark-gray](2,3) rectangle (3,4);
						\filldraw[dark-gray](2,4) rectangle (3,5);
						\filldraw[dark-gray](2,5) rectangle (3,6);
						\filldraw[light-gray](2,6) rectangle (3,7);
						\filldraw[dark-gray](3,0) rectangle (4,1);
						\filldraw[light-gray](3,1) rectangle (4,2);
						\filldraw[dark-gray](3,2) rectangle (4,3);
						\filldraw[dark-gray](3,3) rectangle (4,4);
						\filldraw[light-gray](3,4) rectangle (4,5);
						\filldraw[light-gray](3,5) rectangle (4,6);
						\filldraw[light-gray](3,6) rectangle (4,7);
						\filldraw[dark-gray](4,0) rectangle (5,1);
						\filldraw[dark-gray](4,1) rectangle (5,2);
						\filldraw[dark-gray](4,2) rectangle (5,3);
						\filldraw[light-gray](4,5) rectangle (5,6);
						\filldraw[dark-gray](4,6) rectangle (5,7);
						\filldraw[dark-gray](5,0) rectangle (6,1);
						\filldraw[light-gray](5,2) rectangle (6,3);
						\filldraw[dark-gray](5,3) rectangle (6,4);
						\filldraw[dark-gray](5,4) rectangle (6,5);
						\filldraw[dark-gray](5,5) rectangle (6,6);
						\filldraw[dark-gray](5,6) rectangle (6,7);
						\filldraw[dark-gray](6,1) rectangle (7,2);
						\filldraw[dark-gray](6,2) rectangle (7,3);
						\filldraw[dark-gray](6,3) rectangle (7,4);
						\filldraw[dark-gray](6,4) rectangle (7,5);
						\filldraw[dark-gray](6,5) rectangle (7,6);
						\filldraw[dark-gray](6,6) rectangle (7,7);
						\draw[black, fill=black] (1,4) circle (0.2);
						\draw[black, fill=black] (2,6) circle (0.2);
						\draw[black, fill=black] (3,2) circle (0.2);
						\draw[black, fill=black] (4,5) circle (0.2);
						\draw[black, fill=black] (5,3) circle (0.2);
						\draw[black, fill=black] (6,1) circle (0.2);
						\draw[thick](0,0)--(0,7);
						\draw[thick](0,0)--(7,0);
						\draw[thick](1,0)--(1,7);
						\draw[thick](0,1)--(7,1);
						\draw[thick](2,0)--(2,7);
						\draw[thick](0,2)--(7,2);
						\draw[thick](3,0)--(3,7);
						\draw[thick](0,3)--(7,3);
						\draw[thick](4,0)--(4,7);
						\draw[thick](0,4)--(7,4);
						\draw[thick](5,0)--(5,7);
						\draw[thick](0,5)--(7,5);
						\draw[thick](6,0)--(6,7);
						\draw[thick](0,6)--(7,6);
						\draw[thick](7,0)--(7,7);
						\draw[thick](0,7)--(7,7);
						\draw[ultra thick] (1.2,4.2)--(1.8,5.8);
						\draw[ultra thick] (4.2,4.8)--(4.8,3.2);
						\draw[ultra thick] (1.2,2.8)--(2.8,2.2);
						\node at (0.5, 2.5) {$F$};
					\end{tikzpicture}
				\end{center}
				\caption{}
				\label{figure:simples-case2-3}
			\end{subfigure}
			\begin{subfigure}[b]{0.24\textwidth}
				\begin{center}
					\begin{tikzpicture}[scale=.4]
						\filldraw[dark-gray](0,3) rectangle (1,4);
						\filldraw[dark-gray](0,4) rectangle (1,5);
						\filldraw[dark-gray](0,5) rectangle (1,6);
						\filldraw[dark-gray](1,0) rectangle (2,1);
						\filldraw[dark-gray](1,2) rectangle (2,3);
						\filldraw[dark-gray](1,5) rectangle (2,6);
						\filldraw[dark-gray](2,0) rectangle (3,1);
						\filldraw[dark-gray](2,2) rectangle (3,3);
						\filldraw[dark-gray](2,3) rectangle (3,4);
						\filldraw[dark-gray](2,4) rectangle (3,5);
						\filldraw[light-gray](2,5) rectangle (3,6);
						\filldraw[light-gray](3,0) rectangle (4,1);
						\filldraw[dark-gray](3,1) rectangle (4,2);
						\filldraw[dark-gray](3,2) rectangle (4,3);
						\filldraw[light-gray](3,3) rectangle (4,4);
						\filldraw[light-gray](3,4) rectangle (4,5);
						\filldraw[light-gray](3,5) rectangle (4,6);
						\filldraw[dark-gray](4,0) rectangle (5,1);
						\filldraw[dark-gray](4,1) rectangle (5,2);
						\filldraw[light-gray](4,4) rectangle (5,5);
						\filldraw[dark-gray](4,5) rectangle (5,6);
						\filldraw[light-gray](5,1) rectangle (6,2);
						\filldraw[dark-gray](5,2) rectangle (6,3);
						\filldraw[dark-gray](5,3) rectangle (6,4);
						\filldraw[dark-gray](5,4) rectangle (6,5);
						\filldraw[dark-gray](5,5) rectangle (6,6);
						\filldraw[light-gray](0,2) rectangle (1,3);
						\filldraw[light-gray](5,0) rectangle (6,1);
						\draw[black, fill=black] (1,3) circle (0.2);
						\draw[black, fill=black] (2,5) circle (0.2);
						\draw[black, fill=black] (3,1) circle (0.2);
						\draw[black, fill=black] (4,4) circle (0.2);
						\draw[black, fill=black] (5,2) circle (0.2);
						\draw[thick](0,0)--(0,6);
						\draw[thick](0,0)--(6,0);
						\draw[thick](1,0)--(1,6);
						\draw[thick](0,1)--(6,1);
						\draw[thick](2,0)--(2,6);
						\draw[thick](0,2)--(6,2);
						\draw[thick](3,0)--(3,6);
						\draw[thick](0,3)--(6,3);
						\draw[thick](4,0)--(4,6);
						\draw[thick](0,4)--(6,4);
						\draw[thick](5,0)--(5,6);
						\draw[thick](0,5)--(6,5);
						\draw[thick](6,0)--(6,6);
						\draw[thick](0,6)--(6,6);
						\draw[ultra thick] (1.2,3.2)--(1.8,4.8);
						\draw[ultra thick] (4.2,3.8)--(4.8,2.2);
						\draw[ultra thick] (1.2,1.8)--(2.8,1.2);
						\draw[ultra thick] (0.2,0.2)--(0.8,1.8);
					\end{tikzpicture}
				\end{center}
				\caption{}
				\label{figure:simples-case2-4}
		        \end{subfigure}
		         
		        \caption{Permutation diagrams corresponding to steps in the proof of Theorem \ref{theorem:simples}.}
		\end{figure}
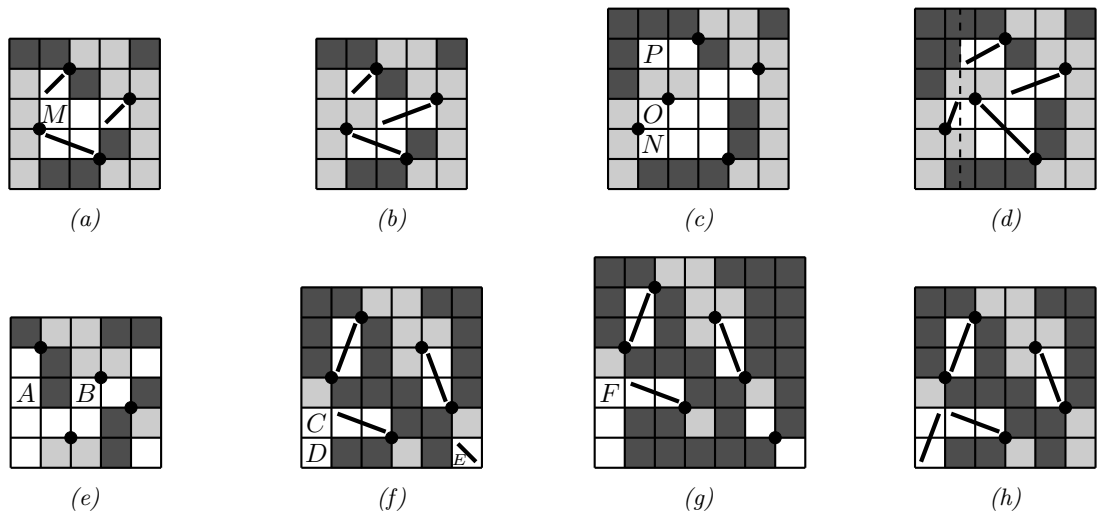	
	
		\textbf{Case 2:} Assume $\pi_R$ contains a $132$ pattern. Let the `$1$' be the bottommost possible entry, let the `$3$' be the topmost possible entry for the chosen `$1$', and let the `$2$' be the rightmost possible entry for the chosen `$1$' and `$3$'. This yields the diagram in Figure~\ref{figure:simples-case2-1}. Any entry in cell $B$ leads to an unsplittable interval inside the rectangular hull formed by any entries in cell $B$ together with the entry just above and to the right of cell $B$. Hence, there must be an entry in cell $A$ to split the rectangular hull of the last two entries shown in the figure. Pick this entry to be as far to the left as possible. To avoid $3124$ and $4312$ patterns, we have the monotone conditions shown in Figure~\ref{figure:simples-case2-2}.
	
	We now show that cell $E$ is empty. Assume it has some entry. Then, the rectangular hull of the leftmost five entries of $\pi$ shown in Figure~\ref{figure:simples-case2-3} must be split in order for $\pi$ to be simple. There is only one place (cell $F$) where we could possibly have a separating element. However, considering the leftmost entry in cell $F$, it is now clear that the rectangular hull of cell $F$ together with the previously considered interval cannot be split. Hence, cell $E$ is empty.
	
	Lastly, any entries in the cells $C$ and $D$ of Figure~\ref{figure:simples-case2-2} together must be increasing, since any descent contained in these two cells would form an unsplittable interval inside of cells $C$ and $D$. So, we have shown that in this case, $\pi$ can be drawn on the diagram in Figure~\ref{figure:simples-case2-4}. Thus, $\pi$ can be drawn on the standard figure of $\GG_1$. 
			
	This completes the proof that $\Simples(\Av(3124,4312)) = \Simples(\GG_1 \cup \GG_2)$.
\end{proof}


\section{The Regular Language and Inflations of $\GG_1$}\label{section:g1}
 
The standard figure for $\GG_1$ is shown in Figure~\ref{figure:our-classes}, along with the directional arrows corresponding to a consistent orientation.

	
Albert, Atkinson, and Vatter \cite[Section 5]{albert:inflations-case-studies} determined the regular languages that are in bijection with $\GG_1$ and $\Simples(\GG_1)$. We repeat this derivation here because it is a good introduction to the following two sections.

Recall the standard notation for regular languages: if $x$ is a letter (or a set of letters), then ``$x^*$'' means zero or more occurrences of $x$ and ``$x^+$'' means one or more occurrences of $x$.

As discussed in Section~\ref{section:ggc-info}, there are two impediments to the bijectivity of the map $\phi : \Sigma^* \to \GG_1$. Firstly, we prevent duplicate words that arise as a result of \emph{commuting pairs} of cells. In this geometric grid class, the set of commuting pairs is $\{(a,c),(a,d),(b,d)\}$. Therefore, to prevent these duplicate words, we forbid all occurrences of $ca$, $da$, and $db$.

Next, we must prevent duplicate words that arise from moving some entry to a different cell. Among all such duplicate words, we choose to prefer the word that has the most entries in the first column, then the most entries in the second column, and then the most entries in the first row.

We define $\LL_1$ to be the regular language consisting of all words $\Sigma^*$, with the following restrictions.
\begin{itemize}
	\item As above, we forbid all words that contain $ca$, $da$, or $db$.
	\item If a word begins with $b$, then the corresponding entry could be moved to cell $a$. Hence, we forbid all words that begin with $b$.
	\item If a word begins with $a^*c$, then entry corresponding to the $c$ could be moved into cell $a$. Hence, we forbid all words that begin with $a^*c$.
	\item If a word ends with $d$, then the entry corresponding to the $d$ could be moved into cell $c$. Hence, we forbid all words that end with $d$.
	\item If a word starts with $d$, has no $c$, and has no other $d$, then the entry corresponding to the $d$ could be moved into cell $c$. Hence, we forbid all words of the form $d\{a,b\}^*$.
	\item If a word is of the form $a^*\{c,d\}^+$, then all entries corresponding to $c$ and $d$ could be moved into cells $a$ and $b$. Hence, we forbid all words of this form. 
\end{itemize}

Then, $\LL_1$ is in (length-preserving) bijection with the geometric grid class $\GG_1$. We compute the multivariate generating function for this regular language (either by hand using the techniques of~\cite[Section I.4]{flajolet:ac}, or using a computer algebra system), in which each letter $a,b,c,d$ is represented by a variable $x_a,x_b,x_c,x_d$ and the coefficient of the term $x_a^{p_1}x_b^{p_2}x_c^{p_3}x_d^{p_4}$ is the number of words in $\LL_1$ that have $p_1$ occurrences of the letter $a$, $p_2$ occurrences of the letter $b$, etc. This multivariate generating function is
	\[ \frac{x_{a}  - x_{a}^{2} - 2  x_{a} x_{c} - 2 x_{a} x_{d} + 2 x_{d}^{2} + 2  x_{a}^{2} x_{c}  + x_{b} x_{c} x_{d} - x_{a}^{2} x_{c}^{2} - 2  x_{a}^{2} x_{c} x_{d} - x_{a}^{2}x_{d}^{2} - x_{b} x_{c} x_{d}^{2}  }{{\left(1 - x_{a}\right)} {\left(1 - x_{c} - x_{d}\right)} {\left(1 - x_{a} - x_{b}- x_{c} - x_{d}  + x_{a} x_{c} + x_{a} x_{d} + x_{b} x_{d} \right)}}.\]
	
We can find the univariate generating function for $\GG_1$ by setting all four variables to $x$. The univariate generating function is:
	\[\frac{ x-4x^2+5x^3 }{ (1-x)(1-2x)(1-3x) }.\]

We will now restrict $\LL_1$ to a new regular language $\calSS_1$ that is in  bijection with $\Simples(\GG_1)$. Permutations that are not simple arise due to either repeated letters (an interval in one cell) or one of the shaded regions involving two or more cells shown in Figure~\ref{figure:g1-simple-regions}. So, we make the following restrictions.

	\begin{figure}
		\begin{center}
			\begin{tikzpicture}[scale=1, baseline=(current bounding box.center)]
				\filldraw[lightgray] (0.5,0.5) rectangle (1.5,1.5);
				\draw[thick] (0,0) rectangle (3,2);
				\foreach \i in {1,2}{
					\draw[thick] (0,\i)--(3,\i);
					\draw[thick] (\i,0)--(\i,2);
				}
				\draw[thick] (3,0)--(3,2);
				\draw[ultra thick,->] (.9,.9)--(.1,.1);
				\draw[ultra thick,<-] (1.9,1.9)--(1.1,1.1);
				\draw[ultra thick,<-] (2.1,1.9)--(2.9,1.1);
				\draw[ultra thick,->] (1.1,.9)--(1.9,.1);
				\node[above left] at (1,0) {$a$};
				\node[above right] at (1,0) {$b$};
				\node[above left] at (2,1) {$c$};
				\node[above right] at (2,1) {$d$};
			\end{tikzpicture}
			\qquad\qquad\qquad
			\begin{tikzpicture}[scale=1, baseline=(current bounding box.center)]
				\filldraw[lightgray] (1.5,1.5) rectangle (2.5,2);
				\draw[thick] (0,0) rectangle (3,2);
				\foreach \i in {1,2}{
					\draw[thick] (0,\i)--(3,\i);
					\draw[thick] (\i,0)--(\i,2);
				}
				\draw[thick] (3,0)--(3,2);
				\draw[ultra thick,->] (.9,.9)--(.1,.1);
				\draw[ultra thick,<-] (1.9,1.9)--(1.1,1.1);
				\draw[ultra thick,<-] (2.1,1.9)--(2.9,1.1);
				\draw[ultra thick,->] (1.1,.9)--(1.9,.1);
				\node[above left] at (1,0) {$a$};
				\node[above right] at (1,0) {$b$};
				\node[above left] at (2,1) {$c$};
				\node[above right] at (2,1) {$d$};
			\end{tikzpicture}
			\qquad\qquad\qquad
			\begin{tikzpicture}[scale=1, baseline=(current bounding box.center)]
				\filldraw[lightgray] (0.5,0.5) rectangle (3,2);
				\draw[thick] (0,0) rectangle (3,2);
				\foreach \i in {1,2}{
					\draw[thick] (0,\i)--(3,\i);
					\draw[thick] (\i,0)--(\i,2);
				}
				\draw[thick] (3,0)--(3,2);
				\draw[ultra thick,->] (.9,.9)--(.1,.1);
				\draw[ultra thick,<-] (1.9,1.9)--(1.1,1.1);
				\draw[ultra thick,<-] (2.1,1.9)--(2.9,1.1);
				\draw[ultra thick,->] (1.1,.9)--(1.9,.1);
				\node[above left] at (1,0) {$a$};
				\node[above right] at (1,0) {$b$};
				\node[above left] at (2,1) {$c$};
				\node[above right] at (2,1) {$d$};
			\end{tikzpicture}
		\end{center}
		\caption{The three regions in which an interval containing entries from at least two cells can occur. This figure also appears in \cite[Figure 12]{albert:inflations-case-studies}.}
		\label{figure:g1-simple-regions}
	\end{figure}
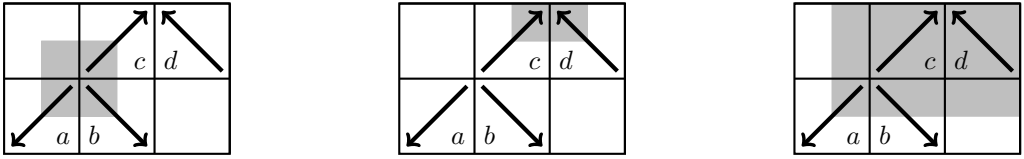
	
	\begin{itemize}
		\item We exclude any words that contain consecutive occurrences of any letter: $aa$, $bb$, $cc$, or $dd$. 
		\item To avoid intervals of the first type, we forbid all words that begin with two or more occurrences of $\{a,b,c\}$. 
		\item To avoid intervals of the second type, we forbid all words that end with $ca^*da^*$ or $d\{a,b\}^*ca^*$.
		\item To avoid intervals of the third type, we require that the last $a$ is followed by a $b$, i.e., we forbid all words that end with $a\{c,d\}^*$.
		\item Lastly, we explicitly forbid the word $dcb$, which does not correspond to a simple permutation but is not forbidden by any of the previous rules.
	\end{itemize}
		
	From these rules, we can find the multivariate generating function of $\calSS_1$:
		\[S_1(x_a, x_b, x_c, x_d) = \frac{ x_{b} x_{c} x_{d}{\left(1 + x_{b}\right)} {\left(x_{a}+ x_{c} + x_{a} x_{c} + x_{c} x_{d} \right)}}{1 - x_{a} x_{b} - x_{b} x_{c} - x_{c} x_{d} - x_{a} x_{b} x_{c} - x_{b} x_{c} x_{d}}.\]
	In particular, this shows that every simple permutation of $\GG_1$ has at least one entry in each of cells $b$, $c$, and $d$. Note that this multivariate generating function excludes the permutation of length 1 and both permutations of length 2. This is because we don't need to consider inflations of the permutation $1$ and we will handle inflations of $12$ and $21$ separately.
	
	We can see that the univariate generating function of $\Simples(\GG_1)$ of length at least 4 is thus:
		\[M_1(x) = \frac{2x^4}{1-2x}.\]
	
	Lastly, we must determine the ways in which we can inflate a simple permutation of $\GG_1$ to yield a permutation in $\Av(3124,4312)$. For the remainder of the paper, we define the two functions
		\[m = \f{x}{1-x}\]
	to be the generating function for nonempty increasing (or decreasing) permutations, and
		\[c = \f{1-2x-\sqrt{1-4x}}{2x}\]
	to be the generating function for nonempty permutations in $\Av(312)$, which are counted by the Catalan numbers. Additionally, we will let $f$ denote the generating function for $\Av(3124,4312)$.
	
	In order to find the allowed inflations of a simple permutation in $\GG_1$, we need to split the letter $c$ into two letters $c_1$ and $c_2$. A $c$ is a $c_1$ if there is any $b$ following it, and a $c_2$ otherwise. By simplicity, there is at most one $c_2$. Additionally, we must split $\Simples(\GG_1)$ into two types of permutations. Let a permutation $\pi \in \Simples(\GG_1)$ be \emph{Type A} if its corresponding word meets the following criteria:
	\begin{itemize}
		\item there is no $c$ after any $b$,
		\item there is no $d$ after any $c$,
		\item the first $b$ is before any $a$.
	\end{itemize}
	
There is one simple permutation of each odd length in $\Simples(\GG_1)$ that meets these criteria (the first few are $25314$, $2475316$, and $246975318$). By construction, there is no $c_2$ in any Type A permutation; every $c$ is a $c_1$. 

	It follows that the multivariate generating function for the Type A permutations in $\Simples(\GG_1)$ is
		\[S_{1,1}\pa{x_a, x_b, x_{c_1}, x_{c_2}, x_d} = \f{x_ax_b^2x_{c_1}x_e}{1-x_ax_b}.\]
	The remainder of the simple permutations of length at least 4 in $\GG_1$ (we will call them \emph{Type B}) thus have multivariate generating function
		\[S_{1,2} = S_1 - S_{1,1}.\]
	
	In Type A permutations, we can inflate any entry in cell $a$ by any permutation in $\Av(312)$. We can inflate any entries in cell $b$ by any decreasing permutation except for the first entry in cell $b$, which can be inflated by any permutation in $\Av(312)$. Any entry in cell $c$ (which must be a $c_1$) may be inflated by any increasing permutation. Entries in cell $d$ may be inflated by permutations in $\Av(312)$. Therefore, the generating function for inflations of simple permutations of Type A that are still in $\Av(3124,4312)$ is
		\[\f{c}{m} \cdot S_{1,1}\pa{c,m,m,0,c}.\]

	In Type B permutations, we consider two cases. In both cases, entries in cell $a$ can be inflated by permutations in $\Av(312)$, entries in cell $b$ can be inflated by decreasing permutations, and the $c_1$ entries in cell $c$ can be inflated by increasing permutations. If a $c_2$ entry in cell $c$ is inflated only by an increasing permutation, then the first entry in cell $d$ may be inflated by a permutation in $\Av(312)$ while all other entries in cell $d$ can only be inflated by decreasing permutations. Otherwise, if the $c_2$ entry is inflated by a permutation with a descent (i.e., a permutation in $\Av(3124,4312) \ssm \Av(21)$), then this forces all entries in cell $d$ to be inflated by only decreasing permutations.
		
	Define $S_{1,3} = S_{1,2} - \pa{S_{1,2}\res{x_{c_2} = 0}}$, so that $S_{1,3}$ is the multivariate generating function of Type B permutations that contain a $c_2$ entry. Now, the generating function for inflations of simple permutations of Type B that are still in $\Av(3124,4312)$ is
		\[\f{c}{m} \cdot S_{1,2}\pa{c,m,m,m,m} + S_{1,3}\pa{c,m,m,f-m,m}.\]
		
	Combining the above results, the (univariate) generating function for the inflations of simple permutations of length at least 4 from $\GG_1$ that lie in $\Av(3124,4312)$ is
		\[I_1(x) = \f{c}{m} \cdot S_{1,1}\pa{c,m,m,0,c} + \f{c}{m} \cdot S_{1,2}\pa{c,m,m,m,m} + S_{1,3}\pa{c,m,m,f-m,m}.\]

\section{The Regular Language and Inflations of $\GG_2$}\label{section:g2}

The standard figure for $\GG_2$ is shown in Figure~\ref{figure:our-classes}, along with the directional arrows corresponding to a consistent orientation. In this geometric grid class, the set of commuting pairs is
	\[\{(a,b),(a,c),(a,d),(a,f),(b,d),(b,e),(b,f),(c,e),(c,f),(d,e),(d,f)\}.\]
	
We note that the three letters $a$, $e$, and $f$ commute with the three letters $b$, $c$, and $d$. So, we forbid all words in which any $b$, $c$, or $d$ comes before any $a$, $e$, or $f$. This handles all commuting pairs except for $(a,f)$ and $(b,d)$, so we further forbid all words containing either $fa$ or $db$.

Next, we must prevent duplicate words that arise from moving some entry to a different cell. Among all such duplicate words, we prefer the word that has the most entries in the first column, then the second column, then the third column, then the first row, then the second row, and then the third row. 

We define $\LL_2$ to be the regular language consisting of all words $\{a,b,c,d,e,f\}^*$, with the following restrictions.

\begin{itemize}
	\item As above, we forbid all words that contain the factor $\{b,c,d\}\{a,e,f\}$ as well as all words containing either $fa$ or $db$.
	\item If a word begins with $e$, then the corresponding entry can be moved to the end of cell $c$. Hence, we forbid all words that start with $e$.
	\item We can move an entry in cell $d$ to cell $b$ if the $d$ has no $c$ or $f$ before it and no $f$ after it. Hence, we forbid all words of the form $\{a,b,d,e\}^*d\{a,b,c,d,e\}^*$.
	\item If a word has no $f$, then any entry in cell $b$ can be moved to cell $a$ (by also moving some entries from cell $c$ to cell $e$ and some entries from cell $d$ to cell $f$, as needed). Hence, we forbid words of the form $\{a,b,c,d,e\}^*b\{a,b,c,d,e\}^*$.
	\item Consider a word that ends in the form $b\{a,e,f\}^*d^*\{a,b,c,e,f\}^*$. Then, the entry in cell $b$ can be moved into cell $a$ (by also moving some entries in cell $c$ to cell $e$ as needed). Hence, we forbid all words that end with $b\{a,e,f\}^*d^*\{a,b,c,e,f\}^*$.
	\item If a word has an $f$ which has no $b$, $c$, or $e$ before it and no $b$ or $c$ after it, then the first $f$ can be moved to cell $c$. Hence, we forbid words of the form $\{a,d,f\}^*f\{a,d,e,f\}^*$.
	\item If a word starts with the prefix $\{a,e,f\}^*c$, then the entry in cell $c$ can be moved to cell $b$. Hence, we forbid words that start with $\{a,e,f\}^*c$.
\end{itemize}

The language $\LL_2$ is in (length-preserving) bijection with the geometric grid class $\GG_2$. We can compute the multivariate generating function for this regular language, but it is too long to display here. We find the univariate generating function for $\GG_2$ by setting all six variables to $x$. The univariate generating function is
	\[\frac{{x-7x^2+19x^3-22x^4+9x^5-x^6}}{{\left(1-x\right)} {\left(1-2x\right)} {\left(1-3x+x^2\right)}^{2}}.\]

We now restrict $\LL_2$ to a new regular language $\calSS_2$ that is in bijection with $\Simples(\GG_2)$ (excluding the permutations $1$, $12$, and $21$). Permutations that are not simple arise due to either repeated letters (an interval in one cell) or one of the shaded regions involving two or more cells shown in Figure~\ref{figure:g2-simple-regions}. So, we make the following restrictions.

	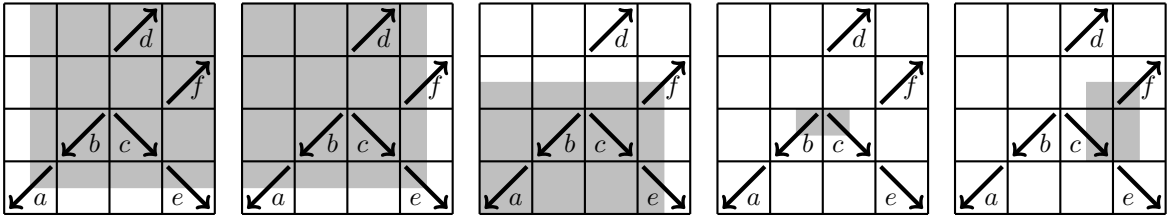
\begin{figure}
		\begin{center}
			\begin{tikzpicture}[scale=.7, baseline=(current bounding box.center)]
				\filldraw[lightgray] (0.5,0.5) rectangle (4,4);
				\foreach \i in {0,1,2,3,4}{
					\draw[thick] (0,\i)--(4,\i);
					\draw[thick] (\i,0)--(\i,4);
				}
				\draw[ultra thick,->] (.9,.9)--(.1,.1);
				\draw[ultra thick,->] (1.9,1.9)--(1.1,1.1);
				\draw[ultra thick,->] (2.1,1.9)--(2.9,1.1);
				\draw[ultra thick,->] (3.1,.9)--(3.9,.1);
				\draw[ultra thick,->] (2.1,3.1)--(2.9,3.9);
				\draw[ultra thick,->] (3.1,2.1)--(3.9,2.9);
				\node[above left] at (1,0) {$a$};
				\node[above left] at (2,1) {$b$};
				\node[above right] at (2,1) {$c$};
				\node[above left] at (3,3) {$d$};
				\node[above left] at (4,2) {$f$};
				\node[above right] at (3,0) {$e$};
			\end{tikzpicture}
			\;
			\begin{tikzpicture}[scale=.7, baseline=(current bounding box.center)]
				\filldraw[lightgray] (0,4) rectangle (3.5,0.5);
				\foreach \i in {0,1,2,3,4}{
					\draw[thick] (0,\i)--(4,\i);
					\draw[thick] (\i,0)--(\i,4);
				}
				\draw[ultra thick,->] (.9,.9)--(.1,.1);
				\draw[ultra thick,->] (1.9,1.9)--(1.1,1.1);
				\draw[ultra thick,->] (2.1,1.9)--(2.9,1.1);
				\draw[ultra thick,->] (3.1,.9)--(3.9,.1);
				\draw[ultra thick,->] (2.1,3.1)--(2.9,3.9);
				\draw[ultra thick,->] (3.1,2.1)--(3.9,2.9);
				\node[above left] at (1,0) {$a$};
				\node[above left] at (2,1) {$b$};
				\node[above right] at (2,1) {$c$};
				\node[above left] at (3,3) {$d$};
				\node[above left] at (4,2) {$f$};
				\node[above right] at (3,0) {$e$};
			\end{tikzpicture}
			\;
			\begin{tikzpicture}[scale=.7, baseline=(current bounding box.center)]
				\filldraw[lightgray] (0,0) rectangle (3.5,2.5);
				\foreach \i in {0,1,2,3,4}{
					\draw[thick] (0,\i)--(4,\i);
					\draw[thick] (\i,0)--(\i,4);
				}
				\draw[ultra thick,->] (.9,.9)--(.1,.1);
				\draw[ultra thick,->] (1.9,1.9)--(1.1,1.1);
				\draw[ultra thick,->] (2.1,1.9)--(2.9,1.1);
				\draw[ultra thick,->] (3.1,.9)--(3.9,.1);
				\draw[ultra thick,->] (2.1,3.1)--(2.9,3.9);
				\draw[ultra thick,->] (3.1,2.1)--(3.9,2.9);
				\node[above left] at (1,0) {$a$};
				\node[above left] at (2,1) {$b$};
				\node[above right] at (2,1) {$c$};
				\node[above left] at (3,3) {$d$};
				\node[above left] at (4,2) {$f$};
				\node[above right] at (3,0) {$e$};
			\end{tikzpicture}
			\;
			\begin{tikzpicture}[scale=.7, baseline=(current bounding box.center)]
				\filldraw[lightgray] (1.5,1.5) rectangle (2.5,2);
				\foreach \i in {0,1,2,3,4}{
					\draw[thick] (0,\i)--(4,\i);
					\draw[thick] (\i,0)--(\i,4);
				}
				\draw[ultra thick,->] (.9,.9)--(.1,.1);
				\draw[ultra thick,->] (1.9,1.9)--(1.1,1.1);
				\draw[ultra thick,->] (2.1,1.9)--(2.9,1.1);
				\draw[ultra thick,->] (3.1,.9)--(3.9,.1);
				\draw[ultra thick,->] (2.1,3.1)--(2.9,3.9);
				\draw[ultra thick,->] (3.1,2.1)--(3.9,2.9);
				\node[above left] at (1,0) {$a$};
				\node[above left] at (2,1) {$b$};
				\node[above right] at (2,1) {$c$};
				\node[above left] at (3,3) {$d$};
				\node[above left] at (4,2) {$f$};
				\node[above right] at (3,0) {$e$};
			\end{tikzpicture}
			\;
			\begin{tikzpicture}[scale=.7, baseline=(current bounding box.center)]
				\filldraw[lightgray] (2.5,1) rectangle (3.5,2.5);
				\foreach \i in {0,1,2,3,4}{
					\draw[thick] (0,\i)--(4,\i);
					\draw[thick] (\i,0)--(\i,4);
				}
				\draw[ultra thick,->] (.9,.9)--(.1,.1);
				\draw[ultra thick,->] (1.9,1.9)--(1.1,1.1);
				\draw[ultra thick,->] (2.1,1.9)--(2.9,1.1);
				\draw[ultra thick,->] (3.1,.9)--(3.9,.1);
				\draw[ultra thick,->] (2.1,3.1)--(2.9,3.9);
				\draw[ultra thick,->] (3.1,2.1)--(3.9,2.9);
				\node[above left] at (1,0) {$a$};
				\node[above left] at (2,1) {$b$};
				\node[above right] at (2,1) {$c$};
				\node[above left] at (3,3) {$d$};
				\node[above left] at (4,2) {$f$};
				\node[above right] at (3,0) {$e$};
			\end{tikzpicture}
		\end{center}
		\caption{The five regions that are not already forbidden by previous rules in which an interval containing entries from at least two cells can occur.}
		\label{figure:g2-simple-regions}
	\end{figure}
	
	\begin{itemize}
		\item We exclude any words that contain consecutive occurrences of any letter: $aa$, $bb$, $cc$, $dd$, $ee$, or $ff$. 
		\item To avoid intervals of the first type, we require that words do not end with $a\{a,b,c,d,f\}^*$. 
		\item To avoid intervals of the second type, we require that words do not end with $e\{b,c,d,e\}^*$.
		\item To avoid intervals of the third type, we forbid words of the form $\{a,b,c,e,f\}^*f\{a,b,c\}^*$.
		\item To avoid intervals of the fourth type, we must exclude any words that begin with \linebreak $\{a,b,c,e,f\}^*b\{a,b,c,e,f\}^*c$.
		\item To avoid intervals of the last type,  we forbid all words of the form $\{a,d,f\}^*f\{a,d,e,f\}^*c\{a,c,e,f\}^*$.
	\end{itemize}
	 	
	From these rules, we can find the multivariate generating function of $\calSS_2$:
		\[\frac{{x_{d} x_{f}\left(1 + x_{c}\right)} {\left( x_{a} x_{e} + x_{c} x_{d} + x_{b} x_{c} x_{d}  - x_{a} x_{b} x_{c} x_{e} - x_{a} x_{c} x_{d} x_{e}- x_{a} x_{b} x_{c} x_{d} x_{e}\right)} }{{\left(1 - x_{b} x_{c} - x_{c} x_{d}  - x_{b} x_{c} x_{d}\right)} {\left(1 - x_{a}x_{e} - x_{e} x_{f}  - x_{a} x_{e} x_{f} \right)}}.\]

	The univariate generating function of $\Simples(\GG_2)$ of length at least 4 is thus
		\[M_2(x) = \frac{{x^4\left(1-x\right)} {\left(2+x\right)} }{{\left(1-x-x^2\right)}^{2}}.\]

	In order to find the allowed inflations of a simple permutation in $\GG_2$, we need to split the letters $c$ and $d$ each into two letters. We say that a $c$ is a $c_2$ if there is no $b$ or $c$ before it and it does not simultaneously have a $d$ both before and after it. It is a $c_1$ otherwise. We say that a $d$ is a $d_2$ if the word has no $e$, the word has at most one $f$, and the $d$ has no $c$ after it. It is a $d_1$ otherwise. By simplicity, a word can have at most one $d_2$. The multivariate generating function for $\calSS_2$ with these new letters will be denoted $S_2\pa{x_a, x_b, x_{c_1}, x_{c_2}, x_{d_1}, x_{d_2}, x_e, x_f}$. 
		

	We handle three separate cases. In all cases, entries corresponding to $a$, $b$, and $c_2$ can be inflated by any permutation in $\Av(312)$, while entries corresponding to $c_1$ and $e$ can be inflated by decreasing permutations and entries corresponding to $d_1$ can be inflated by increasing permutations. The three cases below specify how entries corresponding to $d_2$ and $f$ may be inflated.
	
	In the case that the word has no $d_2$, all entries in cell $f$ can be inflated by permutations in $\Av(312)$. The multivariate generating function for the words in $\calSS_2$ that have no $d_2$ is defined to be \linebreak $S_{2,1}~=~S_2\res{x_{d_2}=0}$.
	
	In the case that the word has a $d_2$ and this $d_2$ is inflated by an increasing permutation, entries in cell $f$ may be inflated by any permutation in $\Av(312)$. If this $d_2$ contains a descent (i.e., is inflated by a permutation in $\Av(3124,4312)\ssm\Av(12)$), then entries in cell $f$ may only be inflated by decreasing permutations. The multivariate generating function for the words in $\calSS_2$ that have a $d_2$ is \linebreak $S_{2,2} = S_2 - S_{2,1}$.
	
	Combining the above results, the (univariate) generating function for the inflations of simple permutations of length at least 4 in $\GG_2$ is
		\[I_2(x) = S_{2,1}\pa{c,c,m,c,m,0,m,c} + S_{2,2}\pa{c,c,m,c,m,m,m,c} + S_{2,2}\pa{c,c,m,c,m,f-m,m,m}.\]

\section{The Regular Language and Inflations of $\GG_3$}\label{section:g3}

The standard figure for $\GG_3$ is shown in Figure~\ref{figure:our-classes}, along with the directional arrows corresponding to a consistent orientation. The set of commuting pairs is $\{(a,c),(a,d),(b,d),(c,d)\}$. Thus, we forbid words that contain $ca$, $da$, $db$, or $dc$.

Next, we must prevent duplicate words that arise from moving some entry to a different cell. Among all such duplicate words, we prefer the word that has the most entries in the first column, then the second column, then the first row, and then the second row.

We construct $\LL_3$ to be the regular language consisting of all words $\{a,b,c,d\}^*$, with the following restrictions.

\begin{itemize}
	\item To conform to the definition of a ``$\bullet$'' entry in a geometric grid class, we forbid all words that contain more than one $d$.
	\item As above, to avoid duplicate permutations due to commuting pairs, we forbid words that contain $ca$, $da$, $db$, or $dc$.
	\item If a word starts with $b$, then the corresponding entry could be moved to cell $a$. Hence, we forbid words that begin with $b$.
	\item If a word has no $d$ and starts with $a^*c$, then the entry corresponding to the $c$ could be moved into cell $a$. Thus, we forbid words of the form $a^*c\{a,b,c\}^*$.
	\item If a word has no $b$ and at least one $c$ or $d$, then the entry corresponding to a $c$ could be moved to cell $a$ or the entry corresponding to a $d$ could be moved to cell $b$. Thus, we forbid all words of the form $\{a,c,d\}^*\{c,d\}\{a,c,d\}^*$.
	\item If a word has no $c$ and has a $d$, then the entry corresponding to the $d$ can be moved into cell $c$. Thus, we forbid words of the form $\{a,b,d\}^*d\{a,b,d\}^*$.
\end{itemize}

The language $\LL_3$ is then in (length-preserving) bijection with the geometric grid class $\GG_3$. We can compute the multivariate generating function for this regular language, but again it is far too long to display here. 
We can find the univariate generating function for $\GG_3$ by setting all four variables to $x$. The univariate generating function is:
	\[\frac{x - 5x^2 + 10x^3 - 8x^4 + x^6}{{\left(1-x\right)}^{2} {\left(1-2x\right)} {\left(1-3x+x^2\right)}}.\]

We will now restrict $\LL_3$ to a new regular language $\calSS_3$ that is in bijection with $\Simples(\GG_3)$. Permutations that are not simple arise due to either repeated letters (an interval in one cell) or one of the shaded regions involving two or more cells shown in Figure~\ref{figure:g3-simple-regions}. So, we make the following restrictions.

	\begin{figure}
		\begin{center}
			\begin{tikzpicture}[scale=1, baseline=(current bounding box.center)]
				\filldraw[lightgray] (0.5,0.5) rectangle (3,3);
				\foreach \i in {0,1,2,3}{
					\draw[thick] (0,\i)--(3,\i);
					\draw[thick] (\i,0)--(\i,3);
				}
				\draw[ultra thick,->] (.9,.9)--(.1,.1);
				\draw[ultra thick,->] (1.1,2.1)--(1.9,2.9);
				\draw[ultra thick,->] (1.1,.9)--(1.9,.1);
				\draw[fill] (2.5,1.5) circle(.08);
				\node[above left] at (1,0) {$a$};
				\node[above right] at (1,0) {$b$};
				\node[above left] at (2,2) {$c$};
				\node[above left] at (3,1) {$d$};
			\end{tikzpicture}
			\qquad\qquad
			\begin{tikzpicture}[scale=1, baseline=(current bounding box.center)]
				\filldraw[lightgray] (0.5,0.5) rectangle (1.5,1);
				\foreach \i in {0,1,2,3}{
					\draw[thick] (0,\i)--(3,\i);
					\draw[thick] (\i,0)--(\i,3);
				}
				\draw[ultra thick,->] (.9,.9)--(.1,.1);
				\draw[ultra thick,->] (1.1,2.1)--(1.9,2.9);
				\draw[ultra thick,->] (1.1,.9)--(1.9,.1);
				\draw[fill] (2.5,1.5) circle(.08);
				\node[above left] at (1,0) {$a$};
				\node[above right] at (1,0) {$b$};
				\node[above left] at (2,2) {$c$};
				\node[above left] at (3,1) {$d$};
			\end{tikzpicture}
		\end{center}
		\caption{The two regions that are not already forbidden by previous rules in which an interval containing entries from at least two cells can occur.}
		\label{figure:g3-simple-regions}
	\end{figure}
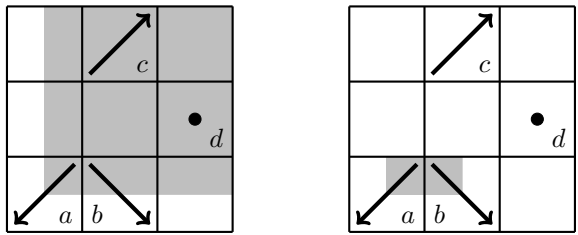

	\begin{itemize}
		\item We exclude any words that contain consecutive occurrences of any letter: $aa$, $bb$, $cc$. (We already can't have $dd$.)
		\item To prevent intervals of the first type, we require that there is a $b$ after the last $a$. Hence, we forbid words that end in $a\{c,d\}^*$.
		\item To prevent intervals of the second type, we forbid words that begin with $ab$.
		\item Lastly, we explicitly forbid the word $cbd$, which does not correspond to a simple permutation but is not forbidden by any previous rule.
	\end{itemize}
	 
	From these rules, we can find the multivariate generating function of $\calSS_3$:
		\[\frac{{x_bx_cx_d\left(1+x_b\right)} {\left(x_a + x_c + x_ax_c\right)}}{1 - x_{a} x_{b} - x_{b} x_{c} - x_{a} x_{b} x_{c}}. \]

	The univariate generating function of $\Simples(\GG_3)$ of length at least 4 is thus
		\[M_3(x) = \frac{{2x^4+x^5}}{1-x-x^2}.\]

	In order to find the allowed inflations of a simple permutation in $\GG_3$, we need to split the letters $b$ and $c$ each into two letters. A $b$ is a $b_2$ if there is no $a$, $b$, or $c_2$ before it and no $c$ after it. It is a $b_1$ otherwise. There is at most one $b_2$. A $c$ is a $c_2$ if there is no $b$ after it. It is a $c_1$ otherwise. By simplicity, a word has at most one $c_2$. The multivariate generating function for $\calSS_3$ using these new letters will be denoted $S_3\pa{x_a, x_{b_1}, x_{b_2}, x_{c_1}, x_{c_2}, x_{d}}$. 

	We handle three separate cases. In all cases, entries corresponding to $a$ and $b_2$ can be inflated by any permutation in $\Av(312)$, while entries corresponding to $b_1$ can be inflated by decreasing permutations and entries corresponding to $c_1$ can be inflated by increasing permutations. The three cases below specify how entries corresponding to $c_2$ and $d$ may be inflated.

	In the case that the word has no $c_2$, an entry in cell $d$ can be inflated by permutations in $\Av(312)$. The multivariate generating function for the words in $\calSS_3$ that have no $c_2$ is defined to be \linebreak $S_{3,1}~=~S_3\res{x_{c_2} = 0}$.
	
	If the word has a $c_2$ and that $c_2$ is inflated by an increasing permutation, then an entry in cell $d$ can be inflated by permutations in $\Av(312)$. If the word has a $c_2$ and that $c_2$ is inflated by a permutation containing a descent (i.e., a permutation in $\Av(3124,4312) \ssm \Av(21)$), then an entry in cell $d$ can only be inflated by decreasing permutations. The multivariate generating function for the words in $\calSS_3$ that have a $c_2$ is $S_{3,2}~=~S_3 - S_{3,1}$.
	
	Combining the above results, the (univariate) generating function for the inflations of simple permutations of length at least 4 in $\GG_3$ is
		\[I_3(x) = S_{3,1}\pa{c,m,c,m,0,c} + S_{3,2}\pa{c,m,c,m,m,c} + S_{3,2}\pa{c,m,c,m,f-m,m}.\]

\section{Computing the Generating Function of $\Av(3124,4312)$}\label{section:final-result}

Theorem~\ref{theorem:simples} proves that $\Simples(\Av(3124,4312)) = \Simples(\GG_1 \cup \GG_2)$. Therefore, we can count the simple permutations in $\Av(3124,4312)$ using the results from Sections~\ref{section:g1}, \ref{section:g2}, and \ref{section:g3}.

\begin{theorem}\label{theorem:simple-enum}
	The simple permutations in $\Av(3124,4312)$ are counted by the generating function
	\[
		S(x) =\frac{x - 2x^2 - 5x^3 + 12x^4 + x^5 - 8x^6 - 3x^7}{{\left(1 - 2x\right)} {\left(1 - x - x^2\right)}^{2}}.
	\]
\end{theorem}
\begin{proof}
We start by counting the permutation of length 1, the two permutations of length 2, the simple permutations in $\GG_1$, and the simple permutations in $\GG_2$. However, this double-counts the simple permutations that lie in both $\GG_1$ and $\GG_2$. Since $\GG_1 \cap \GG_2 = \GG_3$, we correct for this subtracting the generating function for the simple permutations in $\GG_3$. From this we see that the generating function for the simple permutations in $\Av(3124,4312)$ is
	\[S(x) = x + 2x^2 + \pa{M_1(x) + M_2(x) - M_3(x)} = \frac{x - 2x^2 - 5x^3 + 12x^4 + x^5 - 8x^6 - 3x^7}{{\left(1 - 2x\right)} {\left(1 - x - x^2\right)}^{2}},\]
completing the proof.	\end{proof}

The enumeration of $\Av(3214,4312)$ is now derived.

\begin{theorem}\label{theorem:class-enum}
	The permutations in $\Av(3124,4312)$ are counted by the generating function
	\[
		f(x) = \f{\pa{8x^5 - 16x^4 + 28x^3 - 26x^2 + 9x -1} + \sqrt{1-4x}\pa{2x^4-8x^3+14x^2-7x+1}}{2x^2(1-6x+9x^2-4x^3)}.
	\]
\end{theorem}
\begin{proof}
The previous three sections have detailed the allowed inflations of simple permutations of length at least 4. By Lemma~\ref{lemma:decomposition}, it remains to determine the inflations of the permutations $12$ and $21$. To assure uniqueness, we require that the first component in the inflations of $12$ (respectively, $21$) be sum indecomposable (respectively, skew indecomposable).

Let $\pi = \sigma \oplus \tau \in \Av(3124,4312)$ be sum decomposable. We must have $\sigma \in \Av(312)$ and for uniqueness, we assume that $\sigma$ itself is sum indecomposable. For this, we use the notation $\sigma \in \Av_{\not\oplus}(312)$. Then, $\tau$ can be any permutation in $\Av(3124,4312)$. It is well-known that every permutation in $\Av_{\not\oplus}\pa{312}$ is of the form $\alpha \ominus 1$ for $\alpha \in \Av(312)$. Therefore, the class $\Av_{\not\oplus}\pa{312}$ is enumerated by the shifted Catalan numbers, which have the generating function $xc+x$. Now we see that the sum decomposable permutations in $\Av(3124,4312)$ are equal to
	\[\Av_{\not\oplus}(312) \oplus \Av(3124,4312)\]
and are enumerated by
	\[f_{\oplus} = (xc+x)f.\]
	
Let $\pi = \sigma \ominus \tau \in \Av(3124,4312)$ be skew decomposable. There are two possibilities. If $\sigma$ is increasing, then we must have $\tau \in \Av(312)$. Otherwise, if $\sigma \in \Av_{\not\ominus}\pa{4312,3124}$ has a descent, then we must have $\tau \in \Av(12)$. Hence, the skew decomposable permutations in $\Av(3124,4312)$ are equal to
	\[\pa{\Av(21) \ominus \Av(312)} \cup \pa{\Av_{\not\ominus}(3124,4312) \ominus \Av(12)}.\]
We enumerate this class by adding the enumerations of each component of the union and then subtracting the intersection of the two parts (which is $\Av(21) \ominus \Av(12)$). Let $f_{\ominus}$ be the generating function for the skew decomposable permutations in $\Av(3124,4312)$. Then, by the above reasoning,
	\(f_{\ominus} = mc + \pa{f - f_{\ominus}}m - m^2\)
which has the solution
	\[f_{\ominus} = \f{m\pa{f + c - m}}{1+m}.\]

The permutation class $\Av(3124,4312)$ contains the single permutation of length 1, the sum and skew decomposable permutations, and the inflations of simple permutations of length at least 4. Therefore, the generating function $f$ of $\Av(3124,4312)$ satisfies the equation
	\[f = x + (xc+x)f + \f{m\pa{f+c-m}}{1+m} + \pa{I_1 + I_2 - I_3}.\]
We solve this for $f$ to find that
	\[f(x) = \f{\pa{8x^5 - 16x^4 + 28x^3 - 26x^2 + 9x -1} + \sqrt{1-4x}\pa{2x^4-8x^3+14x^2-7x+1}}{2x^2(1-6x+9x^2-4x^3)}.\]
The first few terms of the expansion of $f(x)$ are
	\[f(x) = x + 2x^2 + 6x^3 + 22x^4 + 88x^5 + 363x^6 + 1507x^7 + 6241x^8 + 25721x^9 + 105485x^{10} + \cdots,\]
\OEIS{A165534}.\footnote{We may now compute the number of sum decomposable and skew decomposable permutations in $\Av(3124,4312)$:
	\[f_{\oplus} = \pa{xc+x}f = x^2 + 3x^3 + 10x^4 + 37x^5 + 146x^6 + 595x^7 + 2456x^8 + 10167x^9 + \cdots\]
(\OEIS{A226434}) and
	\[f_{\ominus} = \f{m\pa{f+c-m}}{1+m} = x^2 + 3x^3 + 10x^4 + 35x^5 + 129x^6 + 494x^7 + 1935x^8 + 7670x^9 + \cdots\]
(\OEIS{A228769}).
}\end{proof}

\section{Applicability to Other 2$\times$4 Classes}
\label{section:other-classes}

One may wonder whether these methods apply to any of the nine remaining 2$\times$4 classes that have not yet been enumerated. The key property of $\Av(3124,4312)$ that makes these arguments possible is that its simple permutations lie in a geometric grid class. It is easy to see that a permutation class is not geometrically griddable if it contains either arbitrarily long sums of the permutation $21$ or arbitrarily long skew sums of the permutation $12$.

To this end, we show that each of the remaining 2$\times$4 classes contains a family of simple permutations that contains arbitrarily long sums of $21$ or skew sums of $12$. A natural candidate is the family of increasing oscillations. The \emph{increasing oscillating sequence} is the infinite sequence 
	\[4, 1, 6, 3, 8, 5, \ldots, 2k+2, 2k-1, \ldots,\]
plotted in Figure~\ref{figure:osc-seq}. An \emph{increasing oscillation} is any simple permutation that is contained in the increasing oscillating sequence.  

Brignall, Ru\v{s}kuc, and Vatter \cite{brignall:simple-decidability} showed that the class of all permutations contained in all but finitely many increasing oscillations is $\io$. To show that a class $\CC=\Av(B)$ contains the family of all increasing oscillations, we need to show that $\io \subseteq \CC$. This amounts to checking that each $\beta \in B$ contains some permutation in $\{321,2341,3412,4123\}$. The following seven 2$\times$4 classes contain the family of increasing oscillations and hence their simple permutations are not geometrically griddable:
	\begin{align*}
		&\Av(3214,4231),\quad\Av(1432,4213),\quad\Av(3214,4312),\quad\Av(4231,4321),\\
		&\Av(4123,4231),\quad\Av(3412,4123),\quad\Av(4123,4312).
	\end{align*}
	
This leaves only the two classes $\Av(2143,4213)$ and $\Av(2413,3412)$, which contain neither the family of increasing oscillations nor the analogously defined family of decreasing oscillations. 

	\begin{figure}
		\minipage{0.5\textwidth}
		\begin{center}
			\begin{tikzpicture}[scale=.2]
				\draw[ultra thick] ({1-.07},0)--(12.07,0);
				\draw[ultra thick] (0,{1-.07})--(0,12.07);
				\foreach \x in {1,...,12} {
					\draw[thick] (\x,0)--(\x,-.5);
					\draw[thick] (0,\x)--(-.5,\x);
				}
				\draw[fill=black] (1,4) circle (10pt);
				\draw[fill=black] (2,1) circle (10pt);
				\draw[fill=black] (3,6) circle (10pt);
				\draw[fill=black] (4,3) circle (10pt);
				\draw[fill=black] (5,8) circle (10pt);
				\draw[fill=black] (6,5) circle (10pt);
				\draw[fill=black] (7,10) circle (10pt);
				\draw[fill=black] (8,7) circle (10pt);
				\draw[fill=black] (9,12) circle (10pt);
				\draw[fill=black] (10,9) circle (10pt);
				\node at (10.9,11.3) {$\cdot$};
				\node at (11.5,11.9) {$\cdot$};
				\node at (12.1,12.5) {$\cdot$};
			\end{tikzpicture}
		\end{center}
		\caption{The increasing oscillating sequence, plotted in the style of a permutation plot. Note that there is no index that has entry 2.}
		\label{figure:osc-seq}
		\endminipage\hfill
		\minipage{0.5\textwidth}
		\begin{center}
			\begin{tikzpicture}[scale=.2, baseline=(current bounding box.south)]
				\draw[ultra thick] ({1-.073},0)--(14.073,0);
				\draw[ultra thick] (0,{1-.073})--(0,14.073);
				\foreach \x in {1,...,14} {
					\draw[thick] (\x,0)--(\x,-.5);
					\draw[thick] (0,\x)--(-.5,\x);
				}
				\draw[fill=black] (1,10) circle (10pt);
				\draw[fill=black] (2,1) circle (10pt);
				\draw[fill=black] (3,8) circle (10pt);
				\draw[fill=black] (4,11) circle (10pt);
				\draw[fill=black] (5,9) circle (10pt);
				\draw[fill=black] (6,6) circle (10pt);
				\draw[fill=black] (7,12) circle (10pt);
				\draw[fill=black] (8,7) circle (10pt);
				\draw[fill=black] (9,4) circle (10pt);
				\draw[fill=black] (10,13) circle (10pt);
				\draw[fill=black] (11,5) circle (10pt);
				\draw[fill=black] (12,2) circle (10pt);
				\draw[fill=black] (13,14) circle (10pt);
				\draw[fill=black] (14,3) circle (10pt);
				\node at (15,14.5) {$\cdot$};
				\node at (15.6,14.7) {$\cdot$};
				\node at (16.2,14.9) {$\cdot$};
				\node at (14.5,1.5) {$\cdot$};
				\node at (15.1,1.1) {$\cdot$};
				\node at (15.7,.7) {$\cdot$};
			\end{tikzpicture}
		\end{center}
		\caption{A permutation in an infinite family of simple permutations.}
		\label{figure:other-family}
		\endminipage\hfill
	\end{figure}

Consider instead the family of simple permutations depicted in Figure~\ref{figure:other-family}. No permutation in this family contains either a $2143$ or a $4213$ pattern. Hence this infinite family, which cannot be geometrically gridded, lies in $\Av(2143,4213)$. The last class in question, $\Av(2413,3412)$, does not contain this infinite family, but it is symmetric (by a $90^\circ$ rotation) to $\Av(2143,2413)$ which does contain this infinite family. 

\bigskip
{\bf Acknowledgments:} The author is very grateful to his advisor, Vince Vatter, for introducing him to this problem and for advice that significantly improved the presentation of this paper. Additionally, the author would like to thank Michael Albert for providing code which was useful in developing these arguments. 


\bibliographystyle{acm}
\bibliography{./../../../refs/bib}

\end{document}